\documentclass[11pt,reqno]{amsart}
\usepackage{a4wide}
\usepackage{hyperref}
\usepackage{color}
\usepackage{cite}
\usepackage{amsfonts,amssymb}
\usepackage{pb-diagram}
\usepackage[all]{xy}
\usepackage{graphicx,color}
\usepackage{mathrsfs}
\usepackage{amstext}
\usepackage{sseq}

\definecolor{green}{rgb}{0,0.6,0.2}
\definecolor{blue}{rgb}{0,0,1}

\hypersetup{colorlinks,
linkcolor=blue,
filecolor=green,
citecolor=green}

\theoremstyle{plain}
\newtheorem{theorem}{Theorem}[section]

\newtheorem{prop}[theorem]{Proposition}

\newtheorem{neu}{}[section]
\newtheorem{Cor}[neu]{Corollary}
\newtheorem*{Cor*}{Corollary}
\newtheorem{Thm}[neu]{Theorem}
\newtheorem*{Thm*}{Theorem}
\newtheorem{Prop}[neu]{Proposition}
\newtheorem*{Prop*}{Proposition}
\theoremstyle{definition}

\newtheorem*{Rmk*}{Remark}
\newtheorem{Rmk}[neu]{Remark}

\newtheorem*{Ex*}{Example}
\newtheorem{Qu}[neu]{Question}
\newtheorem*{Qu*}{Question}

\theoremstyle{remark}

\newtheorem{example}[neu]{Example}

\theoremstyle{definition}

\newcommand{\x}{\times}

\newcommand{\wh}{\widehat}

\newcommand{\p}{\partial}
\newcommand{\om}{\omega}
\newcommand{\Om}{\Omega}
\newcommand{\into}{\hookrightarrow}
\newcommand{\pf}{\longrightarrow}

\newcommand{\N}{{\mathbb{N}}}
\newcommand{\Z}{{\mathbb{Z}}}
\newcommand{\R}{{\mathbb{R}}}
\newcommand{\C}{{\mathbb{C}}}
\newcommand{\Q}{{\mathbb{Q}}}

\newcommand{\LLL}{\mathscr{L}}

\newcommand{\ind}{\mathrm{ind}}

\newcommand{\U}{{\mathrm{U}}}

\newcommand{\Sp}{\mathrm{Sp}}

\newcommand{\Spec}{{\rm Spec}}

\newcommand{\Crit}{{\rm Crit}}

\renewcommand{\AA}{\mathcal{A}}

\newcommand{\OO}{\mathcal{O}}

\newcommand{\MM}{\mathcal{M}}

\newcommand{\PP}{\mathcal{P}}

\newcommand{\beq}{\begin{equation}}
\newcommand{\beqn}{\begin{equation}\nonumber}
\newcommand{\eeq}{\end{equation}}

\newcommand{\bea}{\begin{equation}\begin{aligned}}
\newcommand{\bean}{\begin{equation}\begin{aligned}\nonumber}
\newcommand{\eea}{\end{aligned}\end{equation}}

\newcommand{\one}
{{{\mathchoice \mathrm{ 1\mskip-4mu l} \mathrm{ 1\mskip-4mu l}
\mathrm{ 1\mskip-4.5mu l} \mathrm{ 1\mskip-5mu l}}}}

\numberwithin{equation}{section}
\numberwithin{figure}{section}

\begin{document}
\title[Equivariant symplectic homology and multiple closed Reeb orbits]{Equivariant symplectic homology and\\ multiple closed Reeb orbits}
\author{Jungsoo Kang}
\address{Department of Mathematical Sciences\\
     Seoul National University, Seoul, Korea}
\address{Mathematisches Institut,
Westf\"alische Wilhelms-Universit\"at M\"unster, M\"unster, Germany}
\email{jungsoo.kang@me.com}
\begin{abstract}
We study the existence of multiple closed Reeb orbits on some contact manifolds by means of $S^1$-equivariant symplectic homology and the index iteration formula. We prove that a certain class of contact manifolds which admit displaceable exact contact embeddings, a certain class of prequantization bundles, and Brieskorn spheres have multiple closed Reeb orbits.
\end{abstract}

\keywords{$S^1$-equivariant symplectic homology, Index iteration formula, closed Reeb orbit}

\maketitle
\setcounter{tocdepth}{1}
\tableofcontents
\section{Introduction}

After Weinstein's famous conjecture \cite{Wei79}, the existence problem of a closed Reeb orbit has been extensively studied. It is natural to study the multiplicity of (simple) closed Reeb orbits in contact manifolds that are known to have one. This has been addressed in \cite{HWZ98,HWZ03,GHHM12} for tight 3-spheres and in \cite{HT09,CGH12} for general contact 3-manifolds. To the authors knowledge, there are few multiplicity results for general higher dimensional contact manifolds but there are a number of theorems \cite{EL80,BLM85,EH87,LZ02,WHL07,Wan11} for pinched or convex hypersurfaces in $\R^{2n}$.

In the present paper we study the multiplicity problem of closed Reeb orbits for nondegenerate contact manifolds which admit displaceable exact contact embeddings, prequantization bundles, and Brieskorn spheres. Our approach is based on $S^1$-equivariant symplectic homology and the index iteration formula. Although we only treat those three cases, we expect that our method can apply for other contact manifolds for which the formulas of $S^1$-equivariant symplectic homology (or contact homology) are nice in a sense that will be explained below. 

An embedding $i:\Sigma\into W$ of a contact manifold $(\Sigma,\xi)$ into a symplectic manifold $(W,\om)$ is called an {\em exact contact embedding} if $i(\Sigma)$ is bounding, $\om=d\lambda$ for some 1-form $\lambda$, and there exists a contact form $\alpha$ on $(\Sigma,\xi)$ such that $\ker\alpha=\xi$ and $\alpha-\lambda|_\Sigma$ is exact. Throughout this paper  we identify $i(\Sigma)$ with $\Sigma$ and every contact manifold is assumed to be closed. We also tacitly assume that every manifold is connected. Here by bounding we mean that $\Sigma$ separates $W$ into two connected components of which one is relatively compact. We denote by $W_0$ the relatively compact domain. This embedding is said to be {\em displaceable} if there exists a function $F\in C^\infty_c(S^1\x W)$ such that the associated Hamiltonian diffeomorphism $\phi_F$ displaces $\Sigma$ from itself, i.e. $\phi_F(\Sigma)\cap \Sigma=\emptyset$. A symplectic manifold $(W,\om)$ is called {\em convex at infinity} if there exists an exhaustion $W=\bigcup_{k}W_k$ of $W$ by compact sets $W_k\subset W_{k+1}$ with smooth boundaries  such that $\lambda|_{\p W_k}$, $k\in\N$ are contact forms. 

The {\em Reeb vector filed} $R$ on $(\Sigma,\alpha)$ is characterized by $\alpha(R)=1$ and $i_R d\alpha=0$. We recall that a closed Reeb orbit is {\em nondegenerate} if the linearized Poincar\'e return map associated to the orbit has no eigenvalue equal to 1. A contact from $\alpha$ on $(\Sigma,\xi)$ is called nondegenerate if every closed Reeb orbit is nondegenerate. \\[-1.5ex]

\noindent\textbf{Theorem A.} {\em Suppose that a closed contact manifold $(\Sigma,\xi)$ of dimension $2n-1$ admits a displaceable exact contact embedding into $(W,\om)$ which is convex at infinity and satisfies $c_1(W)|_{\pi_2(W)}=0$. Assume that at least one of the following conditions is satisfied:
\begin{itemize}
\item[(i)] $H_*(W_0,\Sigma;\Q)\neq0$ for some $*\in 2\N-1$ 
\item[(ii)] $H_*(W_0,\Sigma;\Q)=0$ for all even degree $*\leq 2n-4$
\end{itemize}
where $W_0$ is the relatively compact domain bounded by $\Sigma$. Then there are at least two closed Reeb orbits contractible in $W$ for any nondegenerate contact form $\alpha$ on $(\Sigma,\xi)$ such that $\alpha-\lambda|_\Sigma$ is exact.  }\\[-1.5ex]

One may ask if there are more than two closed Reeb orbits when both conditions (i) and (ii) are fulfilled. This question does not seem to be easily answered in general. However the Conley-Zehnder index of closed Reeb orbits on  3-dimensional contact manifolds is special enough to answer this question and the precise statement is given below.

The following contact manifolds meet the condition (ii) in the theorem. 
\begin{itemize}
\item[(1)] $(\Sigma,\xi)$ is a rational homology sphere;
\item[(2)] $(\Sigma,\xi)$ is a $\pi_1$-injective fillable 5-manifold;
\item[(3)] $(\Sigma,\xi)$ is a Weinstein fillable $5$-manifold;
\item[(4)] $(\Sigma,\xi)$ is a subcritical Weinstein fillable $7$-manifold.
\end{itemize}
It is worth pointing out that due to \cite[Lemma 3.4]{FSvK12} 
$$
H_*(\Sigma;\Q)\cong H_{*+1}(W_0,\Sigma;\Q)\oplus H_*(W_0;\Q)
$$
if $\Sigma$ is displaceable in $W$ and in particular $H_{*}(\Sigma;\Q)=0$ implies $H_{*+1}(W_0,\Sigma;\Q)=0$. Moreover this equation implies $H_1(W_0,\Sigma;\Q)=0$. \\[-1.5ex]

\begin{Qu*}
{\em Does every (nondegenerate) subcritical Weinstein fillable closed contact manifold has two closed Reeb orbits? More generally, does every closed contact manifold admitting a displaceable exact contact embeddings possess two closed Reeb orbits?}
\end{Qu*}

We expect that the above question will be answered positively. There is no particular reason for conditions (i) and (ii) in Theorem A to be essential. We include some examples in the appendix which do not meet such conditions but have two closed Reeb orbits. 

As we mentioned above, it turns out that every 3-dimensional closed contact manifold has two closed Reeb orbits \cite{CGH12}. Moreover if a nondegenerate closed contact 3-manifold is not a lens space there are at least three closed Reeb orbits \cite{HT09}. In the following we show that if a contact manifold $(\Sigma,\xi)$ in Theorem A is of dimension 3, we have at least $2+b_3(W_0,\Sigma;\Q)$ closed Reeb orbits where $b_3$ denotes the third Betti number. \\[-1.5ex]

\noindent\textbf{Corollary A.} {\em Suppose that a 3-dimensional closed contact manifold $(\Sigma,\xi)$ admits an exact contact embedding into $(W,\om)$ which is convex at infinity and satisfies $c_1(W)|_{\pi_2(W)}=0$. If $\Sigma$ displaceable in $(W,\om)$, then for any nondegenerate contact form $\alpha$ such that $\alpha-\lambda|_\Sigma$ is exact,
$$
\#\{\textrm{closed Reeb orbits contractible in $W$}\}\geq b_3(W_0,\Sigma;\Q)+2.
$$
Moreover, ${b_3(W,\Sigma;\Q)}$-many simple closed Reeb orbits are of Conley-Zehnder index 2.
In particular, if $(W,\om)$ is subcritical Weinstein,
$$
\#\{\textrm{closed Reeb orbits contractible in $W$}\}\geq b_2(\Sigma;\Q)+2.
$$}

\begin{Rmk}
A closed Reeb orbit $\gamma_0$ of Conley-Zehnder index 3 in Corollary \ref{cor:existence of closed Reeb orbit} and ${b_3(W,\Sigma;\Q)}$-many simple closed Reeb orbits $\{\gamma_1,\dots,\gamma_{b_3(W_0,\Sigma;\Q)}\}$ of Conley-Zehnder index 2 in the above corollary have the following nice property. There exist gradient flow lines of the symplectic action functional which connect such closed Reeb orbits with Morse critical points in $(W_0,\Sigma)$. These gradient flow lines can be used to obtain finite energy planes. This will be discussed in the forthcoming paper \cite{FK14}. For instance for subcritical Weinstein fillable contact 3-manifolds, using such finite energy planes, we are able to prove that if any closed Reeb orbit is linked with such $\gamma_i$, then the linking number is always positive.
\end{Rmk}


As a matter of fact, the proof of Theorem A heavily relies on the facts that the positive part of $S^1$-equivariant symplectic homology is periodic, i.e. $\dim SH_*^{S^1,+}(W)=\dim SH_{*+2}^{S^1,+}(W)$ for not small $*\in\N$ and that the positive part of $S^1$-equivariant symplectic homology vanishes for low degrees (condition (ii) in Theorem A guarantees this). In other words, we can find more than one closed Reeb orbits if (the positive part of) the $S^1$-equivariant symplectic homology of a fillable contact manifold is {\em nice} in such a sense. A certain class of prequantization bundles and Brieskorn spheres which are treated below have nice $S^1$-equivariant symplectic homologies and thus, for them, we are able to find more than one closed Reeb orbit.\\[-1.5ex]

Let $(Q,\Omega)$ be a symplectic manifold with an integral symplectic form $\Omega$, i.e. $[\Omega]\in H^2(Q;\Z)$. For each $k\in\N$, there exists a corresponding {\em prequantization bundle} $P$ over $Q$ with $c_1(P)=k[\Omega]$. Due to \cite{BW58}, such a prequantization bundle $(P,\xi:=\ker\alpha_{BW})$ is a contact manifold with a connection 1-form $\alpha_{BW}$. The following theorem proves the existence of two closed Reeb orbits for a certain class of prequantization bundles which naturally arise from the Donaldson's construction, see Remark \ref{rmk:fillability criterion}.\\[-1.5ex]

\noindent\textbf{Theorem B.} {\em Let $(P,\xi)$ be a prequantization bundle over a simply connected closed integral symplectic manifold $(Q,\Omega)$ with $\dim Q=2n-2$ and $c_1(P)=k[\Omega]$ for some $k\in\N$. Suppose that $[\Omega]$ is primitive in $H^2(Q;\Z)$ and $c_1(Q)=c[\Omega]$ for some $|c|>n-1$ and that $(P,\xi)$ admits an exact contact embedding into $(W,\om)$ with $c_1(W)|_{\pi_2(W)}=0$ which is $\pi_1$-injective. Then for a nondegenerate contact form $\alpha$, there are two closed Reeb orbits contractible in $W$ and hence in $P$.}\\[-1.5ex]

More generally, $S^1$-orbibundles over symplectic orbifolds provide more examples of contact manifolds. In particular Brieskorn spheres, one of the simplest examples, are of our interest since the positive part of the  equivariant symplectic homology (contact homology) was already computed in \cite{Ust99}. For $a=(a_0,\dots,a_n)\in\N^{n+1}$, we define
$$
V_\epsilon(a)=\big\{(z_0,\dots,z_n)\in\C^{n+1}\,\big|\,z_0^{a_0}+\dots z_n^{a_n}=\epsilon\big\}
$$
which is singular when $\epsilon=0$. Then a 1-form $\alpha_a=\frac{i}{8}\sum_{j=0}^n a_j(z_jd\bar{z}_j-\bar{z}_jdz_j)$ on $\Sigma_a=V_0(a)\cap S^{2n+1}$ is a contact form. We call $(\Sigma_a,\xi_a:=\ker\alpha_a)$ a {\em Brieskorn manifold}. When $n$ is odd and $a_0\equiv\pm1$ mod 8 and $a_1=\cdots a_n=2$, $\Sigma_a$ is diffeomorphic to $S^{2n-1}$ and called a {\em Brieskorn sphere}. As mentioned, Brieskorn manifolds are generalized example of prequantization bundles. Indeed, all Reeb flows of $(\Sigma_a,\alpha_a)$ are periodic and thus Brieskorn manifolds can be interpreted as principal circle bundles over symplectic orbifolds. Furthermore a Brieskorn manifold is Weinstein fillable and in fact a  filling symplectic manifold is $V_\epsilon$ with $\epsilon\neq0$. We refer to \cite[Section 7.1]{Gei08} for detailed explanation about Brieskorn manifolds.\\[-1.5ex]

\noindent\textbf{Theorem C.} {\em Brieskorn spheres with nondegenerate contact forms have two closed Reeb orbits.}\\[-1.5ex]

Except 3-dimensional displaceable case, we only can find two closed Reeb orbits but we do not think that this lower bound is optimal. For instance, it is interesting to ask: 
\begin{Qu*}
{\em Can one find more than two closed Reeb orbits on Brieskorn spheres with nonperiodic contact forms?}
\end{Qu*}

\subsubsection*{Acknowledgments} { I would like to thank Otto van Koert for fruitful discussion. His comments and suggestions led to Theorems B and C of the present paper. I also thank my advisor Urs Frauenfelder for consistent help. Many thanks to Peter Albers and Universit\"at M\"unster  for their warm hospitality. Finally, I thank the referee for careful reading of the manuscript and the valuable comments. This work is supported by the NRF grant (Nr. 2010-0007669) and by the SFB 878-Groups, Geometry, and Actions.}

\section{$S^1$-equivariant symplectic homology} 

\subsection{Borel type construction}\quad\\[-1.5ex]

$S^1$-equivariant symplectic homology theory was first introduced in \cite{Vit99}. Recently $S^1$-equivariant symplectic homology theory was rigorously studied and written up in \cite{BO09b,BO12b}. In the present paper following \cite{Vit99,BO09a,BO09b} we use the Borel type construction of $S^1$-equivariant (Morse-Bott) symplectic homology and refer to \cite{BO12b} for other constructions, their equivalences, and applications.

Let $(\Sigma,\xi)$ be a contact manifold which admits an exact contact embedding into a symplectic manifold $(W,d\lambda)$ that is convex at infinity. Consider a nondegenerate contact from $\alpha$ on $\Sigma$ such that $\alpha-\lambda|_\Sigma$ is exact. Upto adding a compactly supported exact 1-form to $\lambda$ we can assume that $\alpha=\lambda|_\Sigma$, see \cite[page 253]{CF09}. We denote by $W_0$ the bounded region of $W\setminus\Sigma$. A neighborhood of $\Sigma$ in $W_0$ can be trivialized by the Liouville flow as $(\Sigma\x(1-\epsilon,1],d(r\alpha))$. The symplectic completion of $(W_0,d\lambda)$ is defined by
$$
\widehat W=W_0\cup_{\p W_0} \Sigma\x[1,\infty), \quad\widehat\om=\left\{\begin{array}{ll} d\lambda &\quad\textrm{on}\quad W_0,\\[1ex] d(r\alpha)&\quad\textrm{on}\quad \Sigma\x[1,\infty).\end{array}\right.
$$
We denote by $\widehat\lambda$ a primitive 1-form of $\widehat\om$ which is $\lambda$ on $W_0$ and $r\alpha$ on $\Sigma\x[1,\infty)$.

We choose an almost complex structure $J$ on $W_0$ which is compatible with $\om$ and preserves the contact hyperplane field $\ker\alpha\subset T\Sigma$. We extend this on $\widehat W$ so that $J$ is invariant under the $\R_+$-action and $Jr\p_r=R$ and $JR=-r\p_r$. Such a $J$ will be called admissible. Here $R$ is the Reeb vector field associated to $\alpha$ and we denote by $\varphi^t_R$ the flow of $R$. The {\em Hamiltonian vector field} $X_H$ associated to a Hamiltonian function $H\in C^\infty(\widehat W)$ is defined by $i_{X_H}\widehat\om=-dH$.

Since we have assumed that $(\Sigma,\alpha)$ is nondegenerate, periods of closed Reeb orbits on $(\Sigma,\alpha)$ form a discrete subset $\Spec(\Sigma,\alpha)$ in $\R_+:=(0,\infty)$. We define a {\em family of $S^1$-invariant admissible Hamiltonians} $K_\tau\in C^\infty(\widehat W\x S^{2N+1})$, $\tau\in\R_+\setminus\Spec(\Sigma,\alpha)$ to have the following properties:
\begin{itemize}
\item[(i)] $K_\tau(x,z)=H_\tau(x)-f(z)$ for $(x,z)\in\widehat W\x S^{2N+1}$;
\item[(ii)] $f\in C^\infty(S^{2N+1})$ is Morse-Bott and invariant under the $S^1$-action on $S^{2N+1}$ given by $\theta\cdot z:=e^{2\pi i\theta}z$ for $\theta\in S^1$, $z\in S^{2N+1}\subset\C^{N+1}$;
\item[(iii)] On $W_0$, $H_\tau<0$ and is a $C^2$-small Morse function;
\item[(iv)] On $\Sigma\x[1,\infty)$, $H_\tau(x)=h_\tau(r)$ for some strictly increasing function $h_\tau:[1,\infty)\to\R_+$ satisfying $h_\tau''(r)>0$ on $(1,r_0)$ for some $r_0>0$;
\item[(v)] $h_\tau(r)=\tau r-h_0$ for some $h_0>\tau r_0$ on $\Sigma\x(r_0,\infty)$.
\end{itemize}

With a family of admissible Hamiltonians $K_\tau\in C^\infty(\widehat W\x S^{2N+1})$, we define a family of action functionals $\AA_{K_\tau}^N:\LLL_{\widehat W}\x S^{2N+1}\to\R$, where $\LLL_{\widehat W}$ denotes the space of contractible loops in $\widehat W$, by
$$
\AA_{K_\tau}^N(v,z):=\int_{S^1}v^*\widehat\lambda-\int_{S^1}K_\tau(v,z)dt.
$$
We note that this action functional is $T^2$-invariant with respect to the following torus action on $\LLL_{\widehat W}\x S^{2N+1}$,
$$
(\theta_1,\theta_2)\cdot\big(v(t),z\big):=\big(v(t-\theta_1),e^{2\pi i\theta_2}z\big),\quad (\theta_1,\theta_2)\in T^2,\;t\in S^1,\;z\in S^{2N+1}\subset\C^{N+1}.
$$
That is $\AA_{K_\tau}^N((\theta_1,\theta_2)\cdot(v,\lambda))=\AA_{K_\tau}^N(v,\lambda)$, and thus the critical points set $\Crit\AA^N_{K_\tau}$ is $T^2$-invariant as well. Here $(v,z)\in\Crit\AA^N_{K_\tau}$ if and only if
$$
\left\{\begin{aligned}\frac{d}{dt}v-X_{H_\tau}(v)=0,\\[1ex]
d_zf(z)=0.
\end{aligned}\right.
$$
There are two types of critical points of $\AA^N_{K_\tau}$: 
\begin{itemize}
\item[1)] $(v,z)$ where $z\in\Crit f$ and where $v\equiv x\in W_0$ is a critical point of the Morse function $H_\tau|_{W_0}$;
\item[2)] $(v,z)$ where $z\in\Crit f$ and where $v\in\LLL_{\widehat W}$ lying on levels $\Sigma\x\{r\}$, $r\in(1,r_0)$ is a solution of
\beq\label{eq:nontirivial critical point eq}
\frac{d}{dt} v=h'_\tau(\pi\circ v)R(v).
\eeq
Here $\pi:\Sigma\x[1,\infty)\to[1,\infty)$ is the projection to the second factor.
\end{itemize}
The second type solutions correspond to closed Reeb orbits with period $h'_\tau(\pi\circ v)\in(0,\tau)$. They are transversally nondegenerate (see \cite[Lemma 3.3]{BO09a}), i.e. 
$$
\ker [d\varphi_{R}^{h'_\tau(\pi\circ v)}(v(0))-\one_{T_{v(0)}\widehat W}]=\Big\langle\frac{d}{dt}v(0)\Big\rangle.
$$ 
We define the diagonal $S^1$-action on $\LLL_{\wh W}\x S^{2N+1}$ by
$$
\theta\cdot\big(v(t),z\big):=\big(v(t-\theta),e^{2\pi i\theta}z\big),\quad t\in S^1,\;z\in S^{2N+1}\subset\C^{N+1}.
$$
 Although we will divide out the diagonal $S^1$-action on $\Crit\AA_{K_\tau}^N$, we are still in a Morse-Bott situation due to the presence of another $S^1$-action. See \cite{Bou02,Fra04} for Floer homology in the Morse-Bott situation. Thus we choose an additional Morse-Bott function $q:\Crit\AA_{K_\tau}^N\to\R$ invariant under the diagonal $S^1$-action such that $q/S^1:\Crit\AA_{K_\tau}^N/S^1\to\R$ is Morse. We denote an $S^1$-family of critical points of $q$ containing $(v,z)$ by 
$$
S_{(v,z)}:=\{\theta\cdot(v,z)\,|\,(v,z)\in\Crit q,\;\theta\in S^1\}.
$$

Suppose that $c_1(W)|_{\pi_2(W)}=0$. We define the index
$
\mu:\Crit q\to \Z
$
by
$$
\mu(v,z)=\mu_{CZ}(v)+\ind_{f}(z)+\ind_q(v,z)
$$
where $\mu_{CZ}$ and $\ind$ stand for the Conley-Zehnder index and the Morse index respectively, see \cite{BO09b,BO12a}. In particular if $(v,z)\in\Crit q$ is of the first type, i.e. $v=x\in W_0$,
$$
\mu(v,z)=\mu_{CZ}(v)+\ind_{f}(z)+\ind_q(v,z)=\ind_{H_\tau|_W}(x)-\frac{\dim W}{2}+\ind_{f}(z).
$$
In this case, $q|_{S_{(v,z)}}$ is a function on $S^1$ invariant under the rotation and thus $\ind_q=0$.

We also define a {\em family of $S^1$-invariant admissible compatible almost complex structures} $J=(J_\lambda^t)$, $\lambda\in S^{2N+1}$, $t\in S^1$ such that $J_\lambda^t$ is an admissible compatible  almost complex structure on $(\widehat W,\widehat\om)$ and is $S^1$-invariant, i.e. $J_\lambda^t=J^{t-\theta}_{\theta\lambda}$ for $\theta\in S^1$. Together with a Riemannian metric $g$ on $S^{2N+1}$ invariant under the diagonal $S^1$-action, a metric on $\LLL_{\widehat W}\x  S^{2N+1}$ is defined by
$$
m_{(v,z)}((\xi_1,\zeta_1),(\xi_2,\zeta_2)):=\int_{S^1}\om(\xi_1,J\xi_2)dt+g(\zeta_1,\zeta_2),\quad (\xi_i,\zeta_i)\in T_v\LLL_{\widehat W}\x T_z S^{2N+1}.
$$
A (negative) gradient flow line $(u,y):C^\infty(\R\x S^1,\widehat W)\x C^\infty(\R,S^{2N+1})$ of $\AA_{K_\tau}^N$ with respect to the metric $m$ is a solution of
\beq\label{eq:gradient flow eq2}
\left\{\begin{aligned}\p_s u+J^t_{y(s)}(\p_tu-X_{H_\tau}(u))=0,\\[1ex]
\p_s y+\nabla_gf(y)=0.
\end{aligned}\right.
\eeq
We denote the moduli space of gradient flow lines with $m$ cascades (gradient flow lines of $\AA_{K_\tau}^N$) from $S_{(v_-,z_-)}$ to $S_{(v_+,z_+)}$ for $(v_\pm,z_\pm)\in\Crit q$ by
$$
\widehat\MM_m(S_{(v_-,z_-)},S_{(v_+,z_+)})=\widehat\MM_m(S_{(v_-,z_-)},S_{(v_+,z_+)};K_\tau,q,J,g).
$$
That is, $\widehat\MM_m(S_{(v_-,z_-)},S_{(v_+,z_+)})=\{(\textbf{u,y,t})=((u_i,y_i)_{1\leq i\leq m},(t_i)_{1\leq i\leq m-1}\}$ such that 
\begin{itemize}
\item[(1)] $(u_i,y_i)$s are solutions of \eqref{eq:gradient flow eq2};
\item[(2)] $(u_1,y_1)$ and $(u_m,y_m)$ satisfy
$$
\lim_{s\to-\infty}(u_1(s),y_1(s))\in W^u(S_{(v_-,z_-)};q),\quad \lim_{s\to\infty}(u_m(s),y_m(s))\in W^s(S_{(v_+,z_+)};q)
$$
where $W^s(S_{(v_+,z_+)};q)$ (resp. $W^u(S_{(v_-,z_-)};q)$) is the (un)stable set of a critical manifold $S_{(v_+,z_+)}$ $(S_{(v_-,z_-)})$ of $q$;
\item[(3)] $t_i\in\R_+$ and $(u_i,y_i)$, $i\in\{1,\dots m-1\}$ satisfy
$$
\lim_{s\to-\infty}(u_{i+1}(s),y_{i+1}(s))=\phi_q^{t_i}\big(\lim_{s\to\infty}(u_{i}(s),y_{i}(s))\big)
$$
where $\phi_q^t$ is the negative gradient flow of $q$.
\end{itemize}
We divide out the $\R^m$-action on $\widehat\MM(S_{(v_-,z_-)},S_{(v_+,z_+)})$ defined by shifting cascades $(\textbf{u,y})$ in the $s$-variable. Then we have the moduli space of gradient flow lines with unparametrized cascades denoted by
$$
\MM(S_{(v_-,z_-)},S_{(v_+,z_+)}):=\widehat\MM(S_{(v_-,z_-)},S_{(v_+,z_+)})/\R^m
$$
We note that solutions of \eqref{eq:gradient flow eq2} are equivariant under the diagonal $S^1$-action, that is if $(u,y)$ solves \eqref{eq:gradient flow eq2}, then so does $\theta\cdot(u,y)$. Since $q$ is invariant under the diagonal $S^1$-action as well, the moduli space $\MM(S_{(v_-,z_-)},S_{(v_+,z_+)})$ carries a free $S^1$-action. We denote the quotient by
$$
\MM_{S^1}(S_{(v_-,z_-)},S_{(v_+,z_+)}):=\MM(S_{(v_-,z_-)},S_{(v_+,z_+)})/S^1.
$$
It turns out that this moduli space is a smooth manifold of dimension
$$
\dim\MM_{S^1}(S_{(v_-,z_-)},S_{(v_+,z_+)})=\mu(v_-,z_-)-\mu(v_+,z_+)-1
$$
for a generic $J$. For the detailed transversality analysis we refer to \cite{BO10}. We define the $S^1$-equivariant chain group $SC_*^{S^1,N}(K_\tau)$ by the $\Q$-vector space generated by $S^1$-families of critical points of $\AA_{K_\tau}^N$ of $\mu$-index $*\in\Z$. 
$$
SC_*^{S^1,N}(K_\tau)=\bigoplus_{\stackrel{{(v,z)}\in\Crit q}{\mu(v,z)=*}}\Q\langle S_{(v,z)}\rangle.
$$
The boundary operator $\p^{S^1}:SC_*^{S^1,N}(K_\tau)\to SC_{*-1}^{S^1,N}(K_\tau)$ is defined by
$$
 \p^{S^1}(S_{(v_-,z_-)})=\sum_{\stackrel{{(v_+,z_+)}\in\Crit q}{\mu(v_-,z_-)-\mu(v_+,z_+)=1}}\!\!\!\!\!\#\MM_{S^1}(S_{(v_-,z_-)},S_{(v_+,z_+)}) \;S_{(v_+,z_+)}
$$
where by $\#$ we mean a signed (via the coherent orientations) count of the number of the finite set $\MM_{S^1}(S_{(v_-,z_-)},S_{(v_+,z_+)})$.
Then $\p^{S^1}\circ\p^{S^1}=0$ and thus we are able to define
$$
HF_*^{S^1,N}(K_\tau)=H_*(SC^{S^1,N}(K_\tau),\p^{S^1}).
$$
Taking direct limits, the {\em $S^1$-equivariant symplectic homology} of $(W_0,\om)$ is defined by
$$
SH_*^{S^1}(W_0):=\lim_{N\to\infty}\lim_{\tau\to\infty}HF_*^{S^1,N}(K_\tau)
$$
As the notation indicates, the homology depends only on $(W_0,\om)$. Here the direct limit of $N$ with respect to the embedding $S^{2N-1}\into S^{2N+1}$ is taken as follows. A Morse-Bott function $f_{N_1}\in C^\infty(S^{2N_1-1})$ extends to a Morse-Bott  function $f_{N_2}\in C^\infty(S^{2N_2-1})$ for $N_2>N_1$ so that $f_{N_2}(z,y)=f_{N_1}(z)+|y|^2$ in a tubular neighborhood of $S^{2N_1-1}$ in $S^{2N_2+1}$ where $y$ is the normal coordinate. This induces a direct system over $N\in\N$.  In order to define the negative/positive part of $S^1$-equivariant symplectic homology, we consider
$$
SC_*^{S^1,-,N}(K_\tau)=
\!\!\!\!\!\bigoplus_{\stackrel{{(v,z)}\in\Crit q}{\AA_{K_\tau}^N(v,z)<\epsilon}}\!\!\!\!\!\Q\langle S_{(v,z)}\rangle,\quad SC_*^{S^1,+,N}(K_\tau)=SC_*^{S^1,N}(K_\tau)/SC_*^{S^1,-,N}(K_\tau)
$$
where $\epsilon<\min\Spec(\Sigma,\alpha)$. That is, $SC^{S^1,-,N}$ resp. $SC^{S^1,+,N}$ is generated by type 1) resp. type 2) critical points of $\AA_{K_\tau}^N$, see the property (v) of $K_\tau$. Since the action values decrease along negative gradient flow lines, there exist associated boundary operators $\p^{S^1}_\pm$ induced by $\p^{S^1}$, and hence we are able to define $SH_*^{S^1,\pm}(W_0)$ the negative/positive part of the $S^1$-equivariant symplectic homology of $(W_0,\om)$.

\subsection{Morse-Bott spectral sequence}\quad\\[-1.5ex]

This subsection is devoted to observe that bad orbits do not contribute to $S^1$-equivariant symplectic homology which is certainly expected to be true in $S^1$-equivariant theory. This is clearly true for contact homology and proofs of the present paper may become more transparent if we use contact homology. Nevertheless we use $S^1$-equivariant symplectic homology since contact homology is still problematic due to transversality issues. To see this feature in $S^1$-equivariant symplectic homology, we use a Morse-Bott spectral sequence. We refer to \cite{Fuk96} for detailed explanation about the Morse-Bott spectral sequence. We should mention that this approach was used by \cite{FSvK12} to study the non-existence of a displaceable exact contact embedding of Brieskorn manifolds.

There is a Morse-Bott spectral sequence which converges to $SH_*^{S^1,+}(W_0)$ whose first page $(E^1,d^1)$ is given by
$$
E_{i,j}^1=\bigoplus_{\stackrel{\gamma\in\PP;}{\mu_{CZ}(\gamma)=i}}H_j(\gamma\x_{S^1}ES^1;\OO_\gamma)
$$
where $\OO_\gamma$ is a orientation rational bundle of $\gamma$ and where $\PP$ is the set of nonconstant closed orbits of $X_{H_\tau}$. We note that if $\gamma$ is a $k$-fold cover of a simple closed orbit, $\gamma\x_{S^1}ES^1$ is the infinite dimensional lens space $B\Z_k$.
We recall that parities of Conley-Zehnder indices of all even/odd multiple covers of a simple closed orbits are the same, i.e. 
$$
\mu_{CZ}(\gamma^{2k})\equiv\mu_{CZ}(\gamma^{2\ell}),\quad\mu_{CZ}(\gamma^{2k+1})\equiv\mu_{CZ}(\gamma^{2\ell+1})\quad\textrm{mod 2},\;\;k,\,\ell\in\N.
$$
See \cite{Vit89,Ust99} for instance. A closed orbit $\gamma$ is called {\em bad} if $\gamma=\gamma_0^k$ for a simple closed orbit $\gamma_0$ and some $k\in\N$ (if fact, $k\in 2\N$) and the parity of $\mu_{CZ}(\gamma)$ disagrees with the parity of $\mu_{CZ}(\gamma_0)$. A closed orbit which is not bad is called {\em good}.  If $\gamma$ is a good orbit, the twist bundle $\OO_\gamma$ is trivial and $H_j(B\Z_k;\Q)$ vanishes except degree zero. If $\gamma$ is a bad orbit, $\OO_\gamma$ is the orientation bundle of $B\Z_k$ and $H_j(B\Z_k;\OO_\gamma)$ vanishes for every degree, see \cite{Vit89}. Therefore only good closed orbits contribute to the first page of the Morse-Bott spectral sequence and thus to the positive part of $S^1$-equivariant symplectic homology as well. Note that $E^1_{i,j}=0$ if $j\neq 0$ and hence the Morse-Bott spectral sequence stabilizes at the second page, i.e. $SH^{S^1,N}=H(E^1,d^1)$.

\begin{Rmk}\label{rmk:contribution}
As the Morse-Bott spectral sequence shows, only $(v,z)\in\Crit q$ with $\mu_{CZ}(v)=*$ (i.e. $\ind_{f}(z)=\ind_{q}(v,z)=0$) contributes to $SH_*^{S^1,N}$. 
\end{Rmk}

\subsection{Resonance identity}\quad\\[-1.5ex]

Following \cite{vK05} we define the mean Euler characteristic by
$$
\chi_m(W_0):=\lim_{N\to\infty}\frac{1}{N}\sum_{\ell=-N}^{N}(-1)^\ell\dim SH_\ell^{S^1,+}(W_0)
$$
if the limit exists. The limit exists if $(W_0,\om)$ is homologically bounded, i.e. $\dim SH_\ell^{S^1,+}(W_0)$, $\ell\in\Z$ are uniformly bounded. From the observation of the previous subsection we know the first page of the Morse-Bott spectral sequence converging to $SH^{S^1,+}(W_0)$ is given by
$$
E^1_{i,j}=\left\{\begin{array}{ll}\bigoplus_{\stackrel{\gamma\in\mathfrak{G}}{\mu_{CZ}(\gamma)=i}}\!\!\!\Q,\quad & j=0,\\[1ex]
0 & j\neq0
\end{array}\right.
$$
where $\mathfrak{G}$ is the set of good closed orbits contractible in $W$. Since the mean Euler characteristic of $E^1$ is the same as that of $SH^{S^1,+}(W_0)$, we have
$$
\chi_m(W_0)=\lim_{N\to\infty}\frac{1}{N}\sum_{\gamma\in\mathfrak{G}_N}(-1)^{\mu_{CZ}(\gamma)}
$$
where $\mathfrak{G}_N$ is the set of good closed orbits of Conley-Zehnder indices in $[-N,N]$. Let $\Delta(\gamma)$ be the mean Conley-Zehnder index of $\gamma$ which will be explained in the next section. From $|\mu_{CZ}(\gamma^k)-k\Delta(\gamma)|<n-1$, see \cite{SZ92},
we have  
$$
k\Delta(\gamma)-(n-1)<\mu_{CZ}(\gamma^k)<k\Delta(\gamma)+(n-1),\quad \Delta(\gamma)=\lim_{k\to\infty}\frac{\mu_{CZ}(\gamma^k)}{k}.
$$
Suppose that $\Delta(\gamma)>0$. Then there exist constants $C_1(k),\,C_2(k)\in[-n+1,n-1]$ such that $\mu_{CZ}(\gamma^k)\in[-N,N]$ if and only if 
\beq\label{eq:index ineq}
\max\bigg\{1,\frac{-N+C_1(k)}{\Delta(\gamma)}\bigg\}\leq k\leq\frac{N+C_2(k)}{\Delta(\gamma)}.
\eeq
We recall that nonconstant closed orbits of $X_{H_\tau}$ correspond to closed Reeb orbits after reparametrization and their Conley-Zehnder indices are the same. We abbreviate by $\mathfrak{G}_s$ the set of simple closed Reeb orbits contractible in $W$ whose multiple covers are all good and by $\mathfrak{B}_s$ the set of simple closed Reeb orbits contractible in $W$ whose even multiple covers are bad. Then \eqref{eq:index ineq} implies the following proposition. This idea is essentially identical to \cite{GK10}.

\begin{Prop}\label{prop:resonance identity}
Let $(W_0,\om)$ be homologically bounded. Assume that there are only finitely many simple closed Reeb orbits on $(\Sigma,\alpha)$ and their mean Conley-Zehnder indices are positive. Then we have 
\beq
\chi_m(W_0)=\sum_{\gamma_g\in\mathfrak{G}_s}\frac{(-1)^{\mu_{CZ}(\gamma_g)}}{\Delta(\gamma_g)}+\sum_{\gamma_b\in\mathfrak{B}_s}\frac{(-1)^{\mu_{CZ}(\gamma_b)}}{2\Delta(\gamma_b)}.
\eeq
\begin{proof}
From \eqref{eq:index ineq}, we have
\bean
\chi_m(W_0)&=\lim_{N\to\infty}\frac{1}{N}\sum_{\gamma\in\mathfrak{G}_N}(-1)^{\mu_{CZ}(\gamma)}\\
&=\lim_{N\to\infty}\frac{1}{N}\bigg\{\sum_{\gamma_g\in\mathfrak{G}_{s}}(-1)^{\mu_{CZ}(\gamma_g)}\frac{N}{\Delta(\gamma_g)}+{\sum_{\gamma_b\in\mathfrak{B}_{s}}(-1)^{\mu_{CZ}(\gamma_b)}\frac{1}{2}\frac{N}{\Delta(\gamma_b)}}+O(1)\bigg\}\\
&=\sum_{\gamma_g\in\mathfrak{G}_s}\frac{(-1)^{\mu_{CZ}(\gamma_g)}}{\Delta(\gamma_g)}+\sum_{\gamma_b\in\mathfrak{B}_s}\frac{(-1)^{\mu_{CZ}(\gamma_b)}}{2\Delta(\gamma_b)}.
\eea
\end{proof}
\end{Prop}

\section{Index Iteration formula}

In the present section, we first recall the Conley-Zehnder index of a closed Reeb orbit and then briefly explain how the Conley-Zehnder index varies under iteration. For detailed explanation we refer to Long's book \cite{Lon02}, see also \cite{CZ84,SZ92,Sal99,Gut12}. For the sake of compatibility, we will adopt the notation and terminology of \cite{Lon02}.

Let $\Sp(2n)$ be the space of $2n\x 2n$ symplectic matrices and $\Sp(2n)^*$ be a subset which consists of nondegenerate elements, i.e.
$$
\Sp(2n)^*:=\{M\in\Sp(2n)\,|\,\det(M-\one_{2n})\neq0\}.
$$
We observe that $\Sp(2n)^*=\Sp(2n)^+\cup\Sp(2n)^-$ where
$$
\Sp(2n)^\pm:=\{M\in\Sp(2n)\,|\,\pm\det(M-\one_{2n})>0\}.
$$
An element $M\in\Sp(2n)$ is called {\em elliptic} if the spectrum $\sigma(M)$ is contained in the unit circle $\U:=\{z\in\C\,|\,|z|=1\}$. Since we are interested in the nondegenerate case, i.e. $M\in\Sp(2n)^*$, $\sigma(M)\subset\U\setminus\{1\}$. The {\em elliptic height} of $M$ is defined by the total algebraic multiplicity of all eigenvalues of $M$ in $\U$ and denoted by $e(M)$. On the other hand if $\sigma(M)\cap\U=\emptyset$, i.e. $e(M)=0$, $M$ is called {\em hyperbolic}.

We abbreviate 
$$
\PP(2n,\tau)^*:=\{\Psi:[0,\tau]\to\Sp(2n)\,|\,\Psi(0)=\one_{2n},\;\Psi(\tau)\in\Sp(2n)^*\}.
$$
For $\Psi\in\PP(2n,\tau)^*$, we join $\Psi(\tau)\in\Sp(2n)^\pm$ to 
$$
-\one_{2n} \quad\textrm{or}\quad \mathrm{diag}(2,1/2,-1,\dots,-1)
$$ 
by a path $\psi:[0,1]\to\Sp(2n)^*$. We recall that there exists a continuous map
$$
\rho:\Sp(2n)\pf S^1
$$
which is uniquely characterized by the naturality, the determinant, and the normalization properties. The map $\rho$ induces an isomorphism between fundamental groups $\pi_1(\Sp(2n))$ and $\pi_1(S^1)$. For any path $r:[0,c]\to\Sp(2n)$, we choose a function $\alpha_r:[0,c]\to\R$ such that $\rho(r(t))=e^{i\alpha_r(t)}$. Then the Maslov-type index for a path $\Psi\in\PP(2n,\tau)^*$ is defined by
$$
\mu(\Psi):=\frac{\alpha_\Psi(\tau)-\alpha_\Psi(0)}{\pi}+\frac{\alpha_\psi(1)-\alpha_\psi(0)}{\pi}\in\Z.
$$
In particular, we denote
$$
\Delta(\Psi):=\frac{\alpha_\Psi(\tau)-\alpha_\Psi(0)}{\pi}\in\R.
$$
and call the {\em mean index} of $\gamma$. We remark that since $\Sp(2n)^*$ is simply connected, both $\mu$ and $\Delta$ are independent of the choice of a path $\psi$. 

Now we associate this Maslov-type index to each closed Reeb orbit contractible in a symplectic filling. Let $\gamma$ be a $\tau$-periodic closed Reeb orbit on $(\Sigma,\alpha,\xi)$ contractible in $(W,\om)$. We take a filling disk $\bar\gamma:D^2\to W$ such that $\bar\gamma|_{\p D^2}=\gamma$. Then a symplectic trivialization $\Phi:\bar\gamma^*\xi\to D^2\x\R^{2n-2}$ and the linearized flow $T\phi_R^t(\gamma(0))|_\xi$ along $\gamma$ induce a path of symplectic matrices 
$$
\Psi_\gamma(t):=\Phi(\gamma(t))\circ T\phi_R^t(\gamma(0))|_\xi\circ\Phi^{-1}(\gamma(0)):[0,\tau]\to\Sp(2n-2).
$$
If $\gamma$ is nondegenerate, $\Psi_\gamma\in\PP(2n-2,\tau)^*$ and we are able to define the {\em Conley-Zehnder index} of $\gamma$ by
$$
\mu_{CZ}(\gamma):=\mu(\Psi_\gamma).
$$
The mean Conley-Zehnder index $\Delta(\gamma)$ is also defined as $\Delta(\Psi_\gamma)$. In order to prove our main results we need to study the Conley-Zehnder indices of $\gamma^k$, $k\in\N$ where 
$$
\gamma^k:[0,k\tau]\to\Sigma,\quad \gamma^k(t):=\gamma(t-j\tau)\;\;\textrm{for}\;\; t\in[j\tau,(j+1)\tau],\;\;1\leq j\leq k-1.
$$
We define the {\em $k$-th iteration} $\Psi^k\in\PP(2n-2,k\tau)^*$ of $\Psi\in\PP(2n-2,\tau)^*$ by
$$                    
\Psi^k(t):=\Psi(t-j\tau)\Psi(\tau)^j,\quad t\in[j\tau,(j+1)\tau],\;\;1\leq j\leq k-1
$$
so that $\Psi_{\gamma^k}=\Psi_\gamma^k$ and $\mu_{CZ}(\gamma^k)=\mu(\Psi_\gamma^k)$. 

Let $M_1$ resp. $M_2$ be $2i\x 2i$ resp. $2j\x 2j$ matrix of the square block form as below.
$$
M_1=\left(\begin{array}{cc} A_1& B_1\\
C_1 & D_1 \end{array}\right),\quad
M_2=\left(\begin{array}{cc}  A_2 & B_2\\
C_2 & D_2 \end{array}\right).
$$
The $\diamond$-product of $M_1$ and $M_2$ is a $2(i+j)\x 2(i+j)$ matrix defined by
$$
M_1\diamond M_2:=\left(\begin{array}{cccc} A_1 & 0 & B_1 & 0\\
0 & A_2 & 0 & B_2\\
C_1 & 0 & D_1 & 0\\
0 & C_2 & 0 & D_2  \end{array}\right).
$$
The following symplectic matrices are called {\em basic normal forms}.
\bean
\bullet\quad & D(\pm 2)=\left(\begin{array}{cc} \pm2 &  0\\
0 & \pm 1/2 \end{array}\right),\\[1ex]
\bullet\quad &N_1(\lambda,b)=\left(\begin{array}{cc} \lambda &  b\\
0 & \lambda \end{array}\right),\quad \lambda=\pm1,\;b=\pm1,\,0,\\[1ex]
\bullet\quad &R(\theta)=\left(\begin{array}{cc} \cos\theta & -\sin\theta\\
\sin\theta & \cos\theta \end{array}\right),\quad \theta\in(0,\pi)\cup(\pi,2\pi),\\[1ex]
\bullet\quad &N_2(\theta,B)=\left(\begin{array}{cc} R(\theta) &  B\\[1ex]
0 & R(\theta) \end{array}\right)\;\textrm{for}\; B=\left(\begin{array}{cc} b_1 &  b_2\\
b_3 & b_4  \end{array}\right),\; \theta\in(0,\pi)\cup(\pi,2\pi),\;b_2\neq b_3.
\eea

We note that $D(\pm 2)$ are basic normal forms for eigenvalues outside $\U$ and $N_1$, $R$, and $N_2$ are basic normal forms for eigenvalues in $\U$. Therefore $e(D)=0$, $e(N_1)=e(R)=2$, and $e(N_2)=4$.

The {\em homotopy set} $\Omega(M)$ of $M\in\Sp(2n)$ is defined by
$$
\Omega(M)=\big\{M'\in\Sp(2n)\,\big|\,\sigma(M')\cap\U=\sigma(M)\cap\U,\; \nu_\lambda(M')=\nu_\lambda(M)\;\;\textrm{for all}\;\;\lambda\in\sigma(M)\cap\U\big\}
$$
where
$$
\nu_\lambda(M):=\dim_\C\ker_\C(M-\lambda\one_{2n}).
$$
We denote by $\Omega^0(M)$ the path connected component of $\Omega(M)$ containing $M$.

\begin{Thm}\label{thm:decomposition into basic normal forms}
For $M\in\Sp(2n)$, there exists a path $h:[0,1]\to\Omega^0(M)$ such that
$$
h(0)=M\quad\textrm{and}\quad h(1)=M_1\diamond\cdots\diamond M_k\diamond M_0
$$
where $M_i$'s, $i\in\{1,\dots,k\}$ are basic normal forms for eigenvalues in $\U$ and $M_0$ is either $D(2)^{\diamond\ell}$ or $D(-2)\diamond D(2)^{\diamond (\ell-1)}$ for some $\ell\in\N$. 
\end{Thm}
\begin{proof}
The proof can be found in \cite[Theorem 1.8.10 \& Corollary 2.3.8]{Lon02}
\end{proof}
Since we are interested in the nondegenerate case, i.e. $M\in\Sp(2n)$ with  $\nu_1(M^k)=0$ for all $k\in\N$,  we can exclude the basic normal form $N_1(\lambda,b)$ since
$$
\nu_1(N_1(\lambda,b)^k)\geq 1,\quad\textrm {for some }\; k\in\N.
$$
Moreover,
$$
\nu_1(R(\theta)^k)=2-2\varphi\Big(\frac{k\theta}{2\pi}\Big),\quad
\nu_1(N_2(\theta,B)^k)=2-2\varphi\Big(\frac{k\theta}{2\pi}\Big),\quad k\in\N,
$$
where $\varphi(a)=0$ if $a\in\Z$ and $\varphi(a)=1$ if $a\notin\Z$. Thus $\theta/2\pi$ should be irrational due to the nondegeneracy condition. Therefore in the case at hand, the endpoint of the path $h:[0,1]\to\Omega^0(M)$ in Theorem \ref{thm:decomposition into basic normal forms} is simply
\beq\label{eq:endpoint in homotopy component}
h(1)= R(\theta_1)\diamond\cdots\diamond R(\theta_p)\diamond N_2(\theta_{p+1},B_1)\diamond\cdots\diamond N_2(\theta_{p+q},B_q)\diamond M_0
\eeq
for $\theta_i/2\pi\in(0,1)\setminus\Q$, $i\in\{1,\dots,p+q\}$. Now we are ready to state the following theorem due to \cite{Lon00} which will plays a crucial role.

\begin{Thm}\label{thm:index formula}
Let $\Psi\in\PP(2n,\tau)$ with $\Psi(\tau)^k\in\Sp(2n)^*$ for all $k\in\N$, i.e. $\Psi^k\in\PP(2n,k\tau)^*$, and $h:[0,1]\to\Omega^0(\Psi(\tau))$ such that $h(0)=\Psi(\tau)$ and $h(1)$ is as \eqref{eq:endpoint in homotopy component}. Then the Maslov index of $\Psi^k$ is
$$
\mu(\Psi^k)=\sum_{1\leq i\leq p}\Big(k(P_i-1)+2\Big[\frac{k\theta_i}{2\pi}\Big]+1\Big)+\sum_{1\leq j\leq q}k W_j+\sum_{1\leq o\leq\ell}k Q_o
$$
where $P_i$'s are odd integers and $W_j$'s, $Q_o$'s are integers and they satisfy
$$
\sum_{i} P_i+\sum_j W_j+\sum_o Q_o=\mu(\Psi).
$$
Here, $[a]\in\Z$ is the biggest integer number smaller than or equal to $a\in\R$.
\end{Thm}
\begin{proof}
The proof can be found in \cite{Lon00} or \cite[Chapter 8]{Lon02}.
\end{proof}

\begin{example}\label{ex:iteration formula for 2-dim}
Let $\gamma$ be a simple closed Reeb orbit on a contact manifold of dimension 3 and suppose that all $\gamma^k$'s are nondegenerate. If $\gamma$ is elliptic, $\mu_{CZ}(\gamma)\in2\Z+1$ and
$$
\mu_{CZ}(\gamma^k)=k(\mu_{CZ}(\gamma)-1)+2[k\theta]+1,\;\;\theta\in(0,1)\setminus\Q.
$$
If $\gamma$ is hyperbolic,
$$
\mu_{CZ}(\gamma^k)=k\mu_{CZ}(\gamma).
$$
If $\gamma$ has a negative real Floquet multiplier, $\mu_{CZ}(\gamma)$ is odd. Otherwise, $\gamma$ has a positive real Floquet multiplier and $\mu_{CZ}(\gamma)$ is even, see \cite[Section 8.1]{Lon02}.
\end{example}

\begin{example}\label{ex:iteration formula for 4-dim}
Let $\gamma$ be a simple closed Reeb orbit on a contact manifold of dimension 5 and suppose that all $\gamma^k$'s are nondegenerate. If $\gamma$ is elliptic, either  $\mu_{CZ}(\gamma)\in2\Z$ and 
$$
\mu_{CZ}(\gamma^k)=k(\mu_{CZ}(\gamma)-2)+2[k\theta_1]+2[k\theta_2]+2,\;\;\theta_1,\,\theta_2\in(0,1)\setminus\Q
$$
or
$$
\mu_{CZ}(\gamma^k)=k\mu_{CZ}(\gamma).
$$
If $\gamma$ is hyperbolic,
$$
\mu_{CZ}(\gamma^k)=k\mu_{CZ}(\gamma).
$$
If $\gamma$ is neither elliptic nor hyperbolic, i.e. $e(\gamma)=2$, then
$$
\mu_{CZ}(\gamma^k)=k(\mu_{CZ}(\gamma)-1)+2[k\theta]+1,\quad\theta\in(0,1)\setminus\Q.
$$
\end{example}

One can see that the Conley Zehnder index cannot decrease (resp. increase) under iteration if $\Sigma$ is 3-dimensional and $\mu_{CZ}(\gamma)$ is positive (resp. negative). Unfortunately this does not remain true for higher dimensions. However if the Conley-Zehnder index of a simple closed Reeb orbit is big or small enough, we still have that property. We recall that our contact manifold $(\Sigma,\alpha)$ is of dimension $2n-1$.

\begin{Prop}\label{prop:index nondecreasing}
Let $\gamma$ be a closed Reeb orbit with $\mu_{CZ}(\gamma)\geq n-1$. If every $\gamma^k$, $\forall k\in\N$ is nondegenerate,  we have
$$
\mu_{CZ}(\gamma^k)\leq\mu_{CZ}(\gamma^{k+1}),\quad k\in\N.
$$
\end{Prop}
\begin{proof}
According to Theorem \ref{thm:index formula}, the Conley-Zehnder index of the $k$-fold cover of  $\gamma$ is of the following form.
$$
\mu_{CZ}(\gamma^k)=kr+\sum_{i=1}^{j}2[k\theta_i]+j,\quad r+j=\mu_{CZ}(\gamma)\geq n-1.
$$
Since $j\in\{0,\dots,n-1\}$ and $\theta_i\in(0,1)\setminus\Q$, $r\geq0$ and thus the claim follows directly.
\end{proof}

\begin{Prop}\label{prop:index jump}
Let $\gamma$ be a closed Reeb orbit with $\mu_{CZ}(\gamma)=n+1$. If every $\gamma^k$, $\forall k\in\N$ is nondegenerate, we have
$$
\mu_{CZ}(\gamma^k)+2\leq\mu_{CZ}(\gamma^{k+1}),\quad k\in\N.
$$
Moreover there exists $k_0\in\N$ such that 
$$
\mu_{CZ}(\gamma^{k_0})+2<\mu_{CZ}(\gamma^{k_0+1}).
$$
\end{Prop}
\begin{proof}
The first inequality follows from that $r\geq2$ in the following form again. 
$$
\mu_{CZ}(\gamma^k)=kr+\sum_{i=1}^{j}2[k\theta_i]+j,\quad r+j=n+1
$$
where $1\leq j\leq n-1$. If $r\geq3$, $\mu_{CZ}(\gamma^{k+1})\geq\mu_{CZ}(\gamma^k)+3$ for all $k\in\N$. If $r=2$, we pick $k_0\in\N$ satisfying $[(k_0+1)\theta_i]-[k_0\theta_i]=1$ for some $i\in\{1,\dots,j\}$ so that $\mu_{CZ}(\gamma^{k_0+1})\geq\mu_{CZ}(\gamma^{k_0})+4$.
\end{proof}

\begin{Prop}\label{prop:index decreasing}
If a simple closed Reeb orbit $\gamma$ has $\mu_{CZ}(\gamma)\leq -n$,
$$
\mu_{CZ}(\gamma^k)>\mu_{CZ}(\gamma^{k+1}),\quad \forall k\in\N.
$$
\end{Prop}
\begin{proof}
We note that the Conley-Zehnder index of $\gamma^k$ is
$$
\mu_{CZ}(\gamma^k)=rk+\sum_{i=1}^j 2[k\theta_i]+j, \quad j\in\{0,1,\dots,n-1\}
$$
and $r+j=\mu_{CZ}(\gamma)\leq-n$. Thus $r\leq-n-j<-2j$ and the claim is proved.
\end{proof}


\section{Proofs of the main results}

We recall that the non-vanishing of $SH^{S1,+}_
j(W_0)$ implies the existence of a (possibly
non-simple) closed Reeb orbit on $(\Sigma,\alpha)$ of Conley-Zehnder index $j\in\N$.
\subsection{Displaceable case}\quad\\[-1.5ex]

This subsection is concerned with a contact manifold $(\Sigma,\xi)$ which admits an exact contact embedding into a symplectic manifold $(W,\om)$ which is convex at infinity and $c_1(W)|_{\pi_2(W)}=0$. We continue to assume that a contact form $\alpha$ on $(\Sigma,\xi)$ is nondegenerate and $\alpha=\lambda|_\Sigma$ for some primitive 1-form $\lambda$ of $\om$. Recall that $W_0\subset W$ is the relatively compact domain bounded by $\Sigma$

\begin{Thm}
Suppose that $\Sigma$ is displaceable in $(W,\om)$. Then the $S^1$-equivariant symplectic homology of $(W_0,\om)$ vanishes.
\end{Thm}

The above vanishing theorem can be proved by applying big theorems. Due to \cite{CF09,AF10}, displaceability of $\Sigma$ in $W$ implies vanishing of the Rabinowitz Floer homology of $(W,\Sigma)$. Then using a long exact sequence involving Rabinowitz Floer homology and symplectic (co)homology in \cite{CFO10} and a unit in symplectic cohomology, \cite {Rit10} proved that vanishing of Rabinowitz Floer homology implies vanishing of symplectic homology. Since there exists a spectral sequence converging to $SH^{S^1}(W_0)$ with second page given by
$$
E^2_{i,j}\cong SH_i(W_0)\otimes H_j(\C P^\infty;\Q),
$$
see \cite{Vit99,BO12b}, $SH^{S^1}(W_0)$ vanishes as well provided that $\Sigma$ is displaceable in $W$. We remark that the last argument can be replaced by a different spectral sequence \cite[Section 8b]{Sei08}. However recently a direct relation between leafwise intersections and vanishing of $SH(W_0)$ and $SH^{S^1}(W_0)$ was studied  in \cite{Kan13} (also in \cite{CO08}) in the case that $(W,\om)$ is the completion of Liouville domain, i.e. $(W,\om=\widehat W,\widehat\om)$. Therefore we have a direct proof of the theorem in that case and leave the following question.
\begin{Qu}
{\em If $\Sigma$ is displaceable in $(W,\om)$, is it displaceable in $(\widehat W,\widehat\om)$ as well? }
\end{Qu}

Combining the above theorem with the Viterbo long exact sequence we obtain the following computation which agrees with the contact homology computation \cite{Yau04} in the subcritical Weinstein case. 


\begin{Prop}\label{prop:symplectic homology computation}
If $\Sigma$ is displaceable in $(W,\om)$, we have
$$
SH_*^{S^1,+}(W_0)\cong\!\!\!\ \bigoplus_{i+j=*+n-1}H_i(W_0,\Sigma;\Q)\otimes H_j(\C P^\infty;\Q).
$$
\end{Prop}
\begin{proof}
The $S^1$-equivariant version of the Viterbo long exact sequence is
$$
\cdots\to H^{S^1}_{*+n}(W_0,\Sigma;\Q)\to SH^{S^1}_*(W_0)\to SH^{S^1,+}_*(W_0)\to H^{S^1}_{*+n-1}(W_0,\Sigma;\Q)\to\cdots.
$$
According to the above theorem, $SH^{S^1,+}_*(W_0)\cong H^{S^1}_{*+n-1}(W_0,\Sigma;\Q)$. Since the $S^1$-action on $(W_0,\Sigma)$ is trivial, we have
$$
H^{S^1}_{*+n-1}(W_0,\Sigma;Q)\cong\!\!\!\bigoplus_{i+j=*+n-1}H_i(W_0,\Sigma;\Q)\otimes H_j(\C P^\infty;\Q).
$$
\end{proof}

 We note that the non-vanishing of $SH^{S^1,+}_*(W_0)$ implies the existence of a closed  Reeb orbit  of Conley-Zehnder index $*\in\N$ on $(\Sigma,\alpha)$, see Remark \ref{rmk:contribution}. Hence, the following corollary directly follows from Proposition \ref{prop:symplectic homology computation}. We would like to mention that this result is not new and has been proved in various ways.

\begin{Cor}\label{cor:existence of closed Reeb orbit}
If $(\Sigma,\alpha)$ is displaceable in $(W,\om)$, there exists a closed Reeb orbit $\gamma$ contractible in $W$ such that $\mu_{CZ}(\gamma)=n+1$ where $2n=\dim W$.
\end{Cor}

A direct consequence of Proposition \ref{prop:resonance identity}  and Proposition \ref{prop:symplectic homology computation} is:

\begin{Cor}\label{cor:resonance identity in the displaceable case}
Suppose that $\Sigma$ is displaceable in $(W,\om)$. If $(\Sigma,\alpha)$ has only finitely many closed Reeb orbits and their mean Conley-Zehnder indices are positive,
$$
\frac{1}{2}\sum_{i=1}^{2n}(-1)^{i+n-1} b_{i}(W_0,\Sigma;\Q)=\chi_m(W_0)=\sum_{\gamma_g\in\mathfrak{G}_s}\frac{(-1)^{\mu_{CZ}(\gamma_g)}}{\Delta(\gamma_g)}+\sum_{\gamma_b\in\mathfrak{B}_s}\frac{(-1)^{\mu_{CZ}(\gamma_b)}}{2\Delta(\gamma_b)}.
$$
\end{Cor}

\noindent\textbf{Proof of Theorem A}.\quad\\[-1.5ex]

\noindent\textbf{Case (i).} Suppose that $H_{2\ell-1}(W_0,\Sigma;\Q)\neq0$ for some $\ell\in\N$ and that  $\gamma$ in Corollary \ref{cor:existence of closed Reeb orbit} is the only closed Reeb orbit on $(\Sigma,\alpha)$. Let $\gamma_0$ be the simple closed Reeb orbit such that $\gamma=\gamma_0^k$ for some $k\in\N$. Applying Proposition \ref{prop:symplectic homology computation}, we have
$$
\left\{\begin{array}{ll}
SH^{S^1,+}_{n+1}(W_0)\cong\bigoplus_{i=1}^nH_{2i}(W_0,\Sigma;\Q),\\[1ex]
SH^{S^1,+}_{2\ell-n}(W_0)\cong \bigoplus_{i=1}^\ell H_{2i-1}(W_0,\Sigma;\Q).
\end{array}\right.
$$
Both groups are non-trivial, and thus multiples of $\gamma_0$ represents nonzero homology classes in $SH^{S^1,+}_{n+1}(W_0)$ and $SH^{S^1,+}_{2\ell-n}(W_0)$. However the parity of $n+1$ and the parity of $2\ell-n$ are different. Assume that the parity of $\mu_{CZ}(\gamma_0)$ is different from the parity of $n+1$. The other case follows in the same manner. Then all multiple covers of $\gamma_0$ with Conley-Zehnder index $n+1$ are bad orbits and thus do not contribute to $SH^{S^1,+}_{n+1}(W_0)$. This contradiction implies the existence of a second orbit  geometrically different from $\gamma_0$.\\[-1.5ex]

\noindent\textbf{Case (ii).} Suppose that $H_{2\ell}(W_0,\Sigma;\Q)=0$ for all $0\leq\ell\leq n-2$ and that $H_{2m-1}(W_0,\Sigma;\Q)=0$ for all $m\in\N$. Assume by contradiction that there exists precisely one simple closed Reeb orbit $\gamma_0$ as above. According to Proposition \ref{prop:symplectic homology computation}, we have
\beq\label{eq:homology computation for thm A}
\left\{\begin{array}{ll}
SH^{S^1,+}_{n-1}(W_0)\cong H_{2n-2}(W_0,\Sigma;\Q),\\[1ex]
SH^{S^1,+}_{n-1+2j}(W_0)\cong H_{2n}(W_0,\Sigma;\Q)\oplus H_{2n-2}(W_0,\Sigma;\Q),\quad j\in\N,\\[1ex]
SH^{S^1,+}_{*}(W_0)\cong \{0\},\quad *\in\Z\setminus\{n-3+2j\,|\,j\in\N\}.
\end{array}\right.
\eeq
\noindent\underline{Subcase 1}. If $\mu_{CZ}(\gamma_0)\geq n+1$, $\mu_{CZ}$ nondecreases under iteration, see Proposition \ref{prop:index nondecreasing}, and thus $\mu_{CZ}(\gamma_0)= n+1$. But even in this case, $\mu_{CZ}(\gamma_0^k)+2\leq\mu_{CZ}(\gamma_0^{k+1})$ for all $k\in\N$ and there exists $k_0\in\N$, $\mu_{CZ}(\gamma_0^{k_0})+2<\mu_{CZ}(\gamma_0^{k_0+1})$  due to Proposition \ref{prop:index jump}. This implies that there exist $j\in\N$ such that multiple covers of $\gamma_0$ cannot generate $SH^{S^1,+}_{n-1+2j}(W_0)$. Thus $\gamma_0$ cannot be the only closed Reeb orbit.\\[-1.5ex]

\noindent\underline{Subcase 2}. Suppose that $\mu_{CZ}(\gamma_0)<n-1$ or $\mu_{CZ}(\gamma_0)=n$. Since $SH_{\mu_{CZ}(\gamma_0)}^{S^1,+}(W_0)=0$ according to \eqref{eq:homology computation for thm A}, there exists another closed Reeb orbit $v$ such that either $\gamma_0$ is a boundary of $v$ or $v$ is a boundary of $\gamma_0$. More precisely, either $d_1[v]=[\gamma_0]$ or $d_1[\gamma_0]=[v]$ where $0\neq [v]\in E^1_{\mu_{CZ}(v),0}$, $0\neq [\gamma_0]\in E^1_{\mu_{CZ}(\gamma_0),0}$ and $d_1:E^1\to E^2$ is the boundary operator. This implies that $\mu_{CZ}(v)=\mu_{CZ}(\gamma_0)+1$ or $\mu_{CZ}(v)=\mu_{CZ}(\gamma_0)-1$, see Remark \ref{rmk:contribution}. Since $v$ has to be a good closed Reeb orbit due to the observation in  Subsection 2.2, $v$ is not a multiple cover of $\gamma_0$. \\[-1.5ex]

\noindent\underline{Subcase 3}. We assume that $\mu_{CZ}(\gamma_0)=n-1$. Due to the index iteration formula,
$$
\mu_{CZ}(\gamma_0^k)=kr+\sum_{i=1}^{j}2[k\theta_i]+j,\quad r+j=n-1
$$
for some $\theta_i\in(0,1)\setminus\Q$. Since $0\leq j\leq n-1$, $r\geq0$. If $r=0$,
$$
\mu_{CZ}(\gamma_0^k)=\sum_{i=1}^{n-1}2[k\theta_i]+n-1,\quad \gamma_0\in\mathfrak{G}_s
$$
and $\Delta(\gamma_0)\neq0$. Thus we have the following identity due to Corollary \ref{cor:resonance identity in the displaceable case}.
\beq\label{eq:mean index A}
\sum_{i=1}^{n-1}2\theta_i=\Delta(\gamma_0)=\frac{2}{1+b_{2n-2}(W_0,\Sigma;\Q)}.
\eeq
Since there is no closed Reeb orbit with Conley-Zehnder index $n-2$ or $n$, every multiple cover of $\gamma_0$ with Conley-Zehnder index $n-1$ is a cycle. Therefore according to \eqref{eq:homology computation for thm A},
$$
SH^{S^1,+}_{n-1}(W_0)=\Q\langle\gamma_0,\dots,\gamma_0^{b_{2n-2}(W_0,\Sigma;\Q)}\rangle.
$$
since $\mu_{CZ}(\gamma_0^k)\leq\mu_{CZ}(\gamma_0^{k+1})$ for all $k\in\N$ as observed in Proposition \ref{prop:index nondecreasing}. Moreover we have
$$
\mu_{CZ}(\gamma_0^{b_{2n-2}(W_0,\Sigma;\Q)+1})=n+1.
$$ 
This implies that there exists $i\in\{1,\dots,n-1\}$ such that $[(1+b_{2n-2}(W_0,\Sigma;\Q))\theta_i]=1$. But this contradicts \eqref{eq:mean index A} and the fact that $\theta_i\in(0,1)\setminus\Q$ for all $1\leq i\leq n-1$. 

If $r\geq1$, then there exists $k_0\in\N$ such that 
$$
\mu_{CZ}(\gamma_0^{k_0+1})\geq\mu_{CZ}(\gamma_0^{k_0})+3
$$
which contradicts \eqref{eq:homology computation for thm A}. Hence there exists a closed Reeb orbit geometrically distinct from $\gamma_0$ and this completes the proof.\hfill$\square$\\[-1.5ex]

\noindent\textbf{Proof of Corollary A}. \quad\\[-1.5ex]

According to Proposition \ref{prop:symplectic homology computation}, 
\beq\label{eq:computation cor A}
\left\{\begin{array}{ll}
SH^{S^1,+}_1(W_0)\cong H_2(W_0,\Sigma;\Q),\\[1ex]
SH^{S^1,+}_{2k}(W_0)\cong H_3(W_0,\Sigma;\Q),\\[1ex]
SH^{S^1,+}_{2k+1}(W_0)\cong H_2(W_0,\Sigma;\Q)\oplus H_4(W_0,\Sigma;\Q),
\end{array}\right.
\eeq
for all $k\in\N$. The second equation of \eqref{eq:computation cor A} implies the existence of $b_3(W_0,\Sigma;\Q)$ closed Reeb orbits with Conley-Zehnder index 2, say $\gamma_1,\dots,\gamma_{b_3(W_0,\Sigma;\Q)}$, see Remark \ref{rmk:contribution}. We claim that all $\gamma_i$s are simple. Indeed if $\gamma_i$ is not simple, it has to be a double cover of a simple one of Conley-Zehnder index 1 and thus bad, see Example \ref{ex:iteration formula for 2-dim}. Therefore,
$$
\mu_{CZ}(\gamma_i^k)=2k,\quad i\in\{1,\dots,b_3(W_0,\Sigma;\Q)\},\;k\in\N.
$$
Since $\dim SH_3^{S^1,+}(W_0)\geq1$, there exists another closed Reeb orbit $v$ with $\mu_{CZ}(v)=3$. If $v$ is simple, there exists another closed Reeb orbit to satisfy \eqref{eq:computation cor A}, see Proposition \ref{prop:index jump}. Suppose that $v$ is a multiple cover of a simple one, say $v_0$, and that there is no simple closed Reeb orbit except $v_0$ and $\gamma_i$'s. Then $\mu_{CZ}(v_0)=1$. If $v_0$ is hyperbolic, i.e. $\mu_{CZ}(v_0^k)=k$, only odd multiple covers take into account. Then there exists another simple closed Reeb orbit since $\dim SH_3^{S^1,+}(W_0)=\dim SH_1^{S^1,+}(W_0)+1\geq2$. Suppose that $v_0$ is elliptic, i.e. $\mu_{CZ}(v_0^k)=2[k\theta]+1$ for some $\theta\in(0,1)\setminus\Q$, see Example \ref{ex:iteration formula for 2-dim}. Since $v_0$ generates all odd degrees of $SH^{S^1,+}_{*}(W_0)$ and $\mu_{CZ}(v_0^k)\leq\mu_{CZ}(v_0^{k+1})$ for all $k\in\N$, 
$$
\mu_{CZ}(v_0^{kb_2(W_0,\Sigma;\Q)+1})=2k+1 ,\quad\forall k\in\N.
$$
Since by the index iteration formula
$$
\mu_{CZ}(v_0^{kb_2(W_0,\Sigma;\Q)+1})=2[(kb_2(W_0,\Sigma;\Q)+1)\theta]+1,
$$
we have 
$$
2[(kb_2(W_0,\Sigma;\Q)+1)\theta]+1=2k+1.
$$
By dividing both sides by $k$ and taking a limit $k\to\infty$, 
$$
\theta=\frac{1}{b_2(W_0,\Sigma;\Q)}\in\Q
$$
and thus $v_0^{b_2(W_0,\Sigma;\Q)}$ is degenerate. This contadiction completes the proof.
\hfill$\square$

\subsection{Prequantization bundles}\quad\\[-1.5ex]

Let $(Q,\Omega)$ be a closed symplectic manifold with an integral symplectic form $\Omega$, i.e. $[\Omega]\in H^2(Q;\Z)$. Since the first Chern class classifies isomorphism classes of complex line bundles, we can find a principal $S^1$-bundle $p:P\to Q$ with $c_1(P)=k[\Omega]$ for $k\in\N$. Such a {\em prequantization bundle} $P$ carries a connection 1-form $\alpha_{BW}$ such that the curvature form of $\alpha_{BW}$ is $-2k\pi\Om$, i.e. $-2k\pi p^*\Omega=d\alpha_{BW}$, see \cite{BW58} or \cite[Chapter 7.2]{Gei08}. Therefore a prequantization bundle $(P,\xi:=\ker\alpha_{BW})$ is a contact manifold and the Reeb flow of $\alpha_{BW}$ is periodic. Suppose that $c_1(Q)=c[\Omega]$ for some $c\in\Z$.
Due to the Gysin sequence for $S^1\into P\stackrel{p}{\to} Q$, $0=p^*c_1(P)=kp^*[\Om]$ and thus $c_1(\xi)=\pi^* c_1(Q)=cp^*[\Om]$ is a torsion class. Hence the Maslov indices for homologically trivial Reeb orbits are well defined. We remark that the generalized Maslov index due to \cite{RS93}  is well defined although the Conley-Zehnder index is not since $(P,\alpha_{BW})$ is Morse-Bott. These two indices agree in the nondegenerate case. Suppose furthermore that $(Q,\omega)$ is simply connected and that $[\omega]$ is a primitive element in $H^2(Q)$. We denote by $\gamma$ a principal orbit in $P$. Then since $\pi_1(P)=\Z_k$, the $k$-fold cover of $\gamma$ is contractible and its Maslov index equals to $2c$, i.e. $\mu_{Maslov}(\gamma^k)=2c$, see \cite[p.100]{Bou02}.

We learned the following remark and proposition from Otto van Koert. 

\begin{Rmk}[\cite{vK12}]\label{rmk:fillability criterion}
In this remark we construct some examples which meet requirements in Theorem B. Let $(B,\om)$ be a simply connected closed integral symplectic manifold such that $[\om]\in H^2(B;\Z)$ is primitive and $c_1(B)=a[\om]$ for some $a\in\Z$. Let $Q_k$ be a symplectic Donaldson hypersurface in $(B,\om)$ Poincar\'e dual to $k[\omega]$ for sufficiently large $k\in\N$, see \cite{Don96} and \cite[Section 6]{CDvK12}. Then according to \cite[Proposition 11]{Gir02}, $W:=B-\nu_B(Q_k)$ is a compact Weinstein domain. Here $\nu_B(Q_k)$ is the the normal disk bundle over $Q_k$ in $B$ with $c_1(\nu_B(Q_k))=k[\om|_{Q_k}]$. Therefore the prequantization  bundle $(P,\alpha_{BW})$ over $Q_k$ with $c_1(P)=k[\om|_{Q_k}]$ has a Weinstein filling $(W,\om|_W)$. Now we show that this example meets the assumptions in Theorem B. 
\begin{itemize}
\item[(i)] $c_1(W)|_{\pi_2(W)}=0$
\end{itemize}
since for $O\in\pi_2(B)$, $\langle c_1(W),O\rangle=\langle c_1(B)|_W,O\rangle=\langle a[\omega|_W],O\rangle=\langle a[d\lambda],O\rangle=0$ for some 1-form $\lambda$ on a Weinstein manifold $W$.
\begin{itemize}
\item[(ii)] $Q_k$ is simply connected
\end{itemize}
by (an analogue of) the Lefschetz hyperplane theorem, see \cite[Proposition 39]{Don96}.
\begin{itemize}
\item[(iii)] $c_1(Q_k)=(a-k)[\om|_{Q_k}]$
\end{itemize}
 due to $c_1(Q_k)=c_1(B)-c_1(\nu_B(Q))$. Moreover if $\dim W\geq6$, we have
\begin{itemize}
\item[(iv)]  $\pi_1(W)\cong\pi_1(\p W)$ 
\end{itemize}
since $W\simeq (\p W\x[0,1])\cup\{k\textrm{-handles}:k\geq 3\}$.

\end{Rmk}

\begin{Prop}[\cite{vK12}]\label{prop:SH for prequantization bundles}
Let $(P,\xi=\ker\alpha_{BW})$ be a prequantization bundle over a simply connected integral symplectic manifold $(Q,\om)$ of dimension $(2n-2)$ such that $[\om]$ is primitive and $c_1(P)=k[\om]$ for $k\in\N$. Suppose that $c_1(Q)=c[\om]$ for some $|c|>n-1$ and that $(P,\xi)$ admits an exact contact embedding $i:(P,\xi)\into(W,d\lambda)$ with $c_1(W)|_{\pi_2(W)=0}$ which is $\pi_1$-injective. Then
$$
SH^{S^1,+}_*(W_0)
\cong
\bigoplus_{N=1}^\infty H_{*-(2Nc-n+1)}(Q;\Q)
$$
where $W_0\subset W$ is the relatively compact domain bounded by $i(P)$.
\end{Prop}

\begin{proof}
We compute the symplectic homology for $(W,d\lambda,P,\alpha_{BW})$, i.e. $\lambda|_{i(P)}=\alpha_{BW}$, and the resulting homology is an invariant for $(W,d\lambda,P,\xi)$. Since each fiber of $P\to Q$ is a closed Reeb orbit,  there is a $P$-family of simple closed Reeb orbits. More generally, the space of $j$-fold covered closed Reeb orbits can be identified with $j$-fold covered fibers which we denote by $P_j$ for each $j\in\N$. We note that since the map $i$ is injective on $\pi_1$-level and only $kN$-fold covered closed Reeb orbits are contractible,  Morse-Bott components are exactly $\{P_{kN}\,|\,k\in\N\}$. Recall that $\gamma$ is a principal orbit and thus $\gamma^{j}\in P_{j}$ for all $j\in\N$.  As in the subsection 2.2, there exists a Morse-Bott spectral sequence with $E^1$-page 
$$
E^1_{p,q}=\bigoplus_{\stackrel{N\in\N;}{\mu_{CZ}(\gamma^{kN})=p}} H^{S^1}_q(P_{kN};\Q)
$$
converging to $SH^{S^1,+}_*(W_0)$.\footnote{Although we constructed $S^1$-equivariant symplectic homology only in the nondegenerate case, the construction still works in the general Morse-Bott case with minor modifications.} 
 We observe that 
$$
\mu_{CZ}(\gamma^{kN})=\mu_{Maslov}(\gamma^{kN})-\frac{1}{2}\dim Q=2Nc-(n-1),
$$
see \cite{Bou02} or \cite[Appendix]{CF09}. Since $P$ is a principal circle bundle and every contractible closed Reeb orbit is good, we have 
$$
H^{S^1}_*(P_{kN};\Q)\cong H_*(Q\x B\Z_{kN};\Q)\cong H_*(Q;\Q).
$$
Since $2c\geq 2n$ and the height of the spectral sequence is $\dim Q=2n-2$, the spectral sequence stabilizes at the $E^1$-page, see Figure~\ref{fig:MB_spectral_sequence}, and hence we conclude
$
SH^{S^1,+}_*(W_0)
\cong
\bigoplus_{N\in\N} H_{*-(2Nc-n+1)}(Q;\Q).
$

\begin{figure}[htb]
\includegraphics[width=0.47\textwidth,clip]{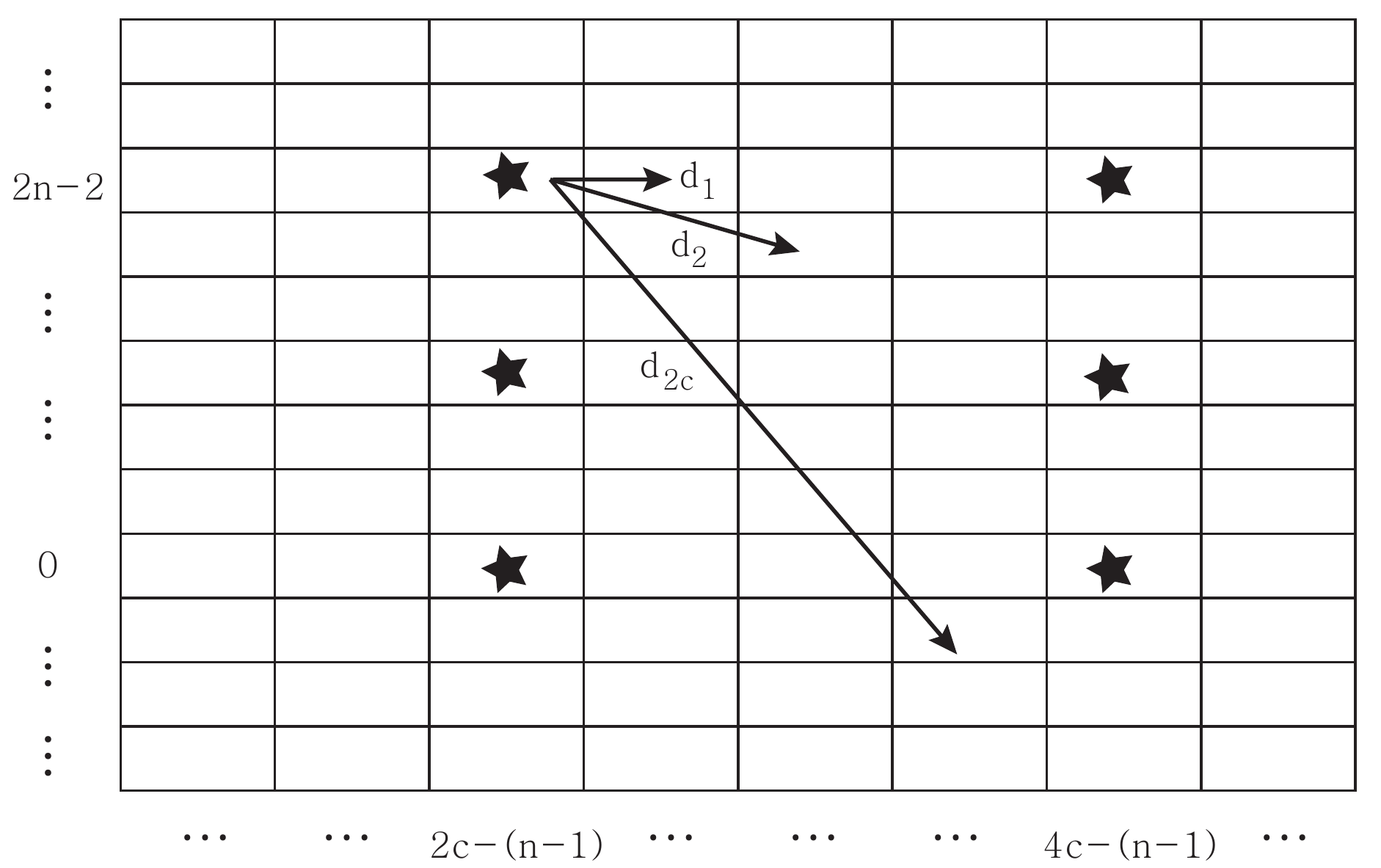}
\caption{$E^1$-page of the Morse-Bott spectral sequence}\label{fig:MB_spectral_sequence}
\end{figure}

\end{proof}

\noindent\textbf{Proof of Theorem B}.\quad\\[-1.5ex]

By Proposition \ref{prop:SH for prequantization bundles}, for all $N\in\N$, 
\begin{equation}\label{eq:computation}
SH^{S^1,+}_{2Nc-(n-1)}(W_0)=H_0(Q;\Q)=\Q,\quad SH^{S^1,+}_{2Nc+(n-1)}(W_0)=H_{2n-2}(Q;\Q)=\Q.
\end{equation}
We first treat the case $c\geq n$. Note that $SH^{S^1,+}_*(W_0)=0$ for all $*<2c-(n-1)$. Assume by contradiction that there is precisely one simple  closed Reeb orbit $\gamma$. If $\mu_{CZ}(\gamma)$ is smaller than $2c-(n-1)$, there has to exist another closed Reeb orbit $v$ with $\mu_{CZ}(v)\in\{\mu_{CZ}(\gamma)+1,\mu_{CZ}(\gamma)-1\}$ such that either $v$ is a boundary of $\gamma$ or $\gamma$ is a boundary of $v$ since $SH^{S^1,+}_{\mu_{CZ}(\gamma)}(W_0)=0$. But $v$ cannot be a multiple cover of $\gamma$ since otherwise $v$ is a bad orbit which is excluded in the $E^1$-page of the Morse-Bott spectral sequence, see Subsection 2.2. We may assume that $\mu_{CZ}(\gamma)\geq 2c-(n-1)\geq n+1$. Then since $\mu_{CZ}(\gamma^{k+1})\geq\mu_{CZ}(\gamma^k)+2$ for all $k$, see Proposition \ref{prop:index jump}, $\mu_{CZ}(\gamma)$ has to be $2c-(n-1)$ the first degree when $SH^{S^1,+}$ does not vanish. 

We first exclude the case $\mu_{CZ}(\gamma^k)=k\mu_{CZ}(\gamma)$, $\forall k\in\N$. Indeed, if this holds, for some $k\in\N$,
 $$
k(2c-(n-1))=\mu_{CZ}(\gamma^k)=2c+(n-1).
$$
If $k\geq3$, $c\leq n-1$ and this contradicts our assumption. Suppose that $k=2$ and $c=(3n-3)/2$. We know that for some $\ell\in\N$, 
$$
\mu_{CZ}(\gamma^\ell)=4c-(n-1)=5n-5.
$$
But on the other hand
$$
\mu_{CZ}(\gamma^\ell)=\ell\mu_{CZ}(\gamma)=\ell(2c-(n-1))=\ell(2n-2).
$$ 
This implies $\ell\notin\N$ and this contradiction excludes the case $\mu_{CZ}(\gamma^k)=k\mu_{CZ}(\gamma)$, $\forall k\in\N$. Therefore the index iteration formula for $\gamma$ has to be of the following form. Note that $\gamma$ is good since it is simple.
$$
\mu_{CZ}(\gamma^k)=rk+\sum_{i=1}^j 2[k\theta_i]+j, \quad j\in\{1,\dots,n-1\}
$$
where $\theta_i\in(0,1)\setminus\Q$ for all $i$. In particular we have
$$
\mu_{CZ}(\gamma)=r+j=2c-(n-1)
$$
Due to \eqref{eq:computation}, there exist $k\in\N$ satisfying 
$$
\mu_{CZ}(\gamma^k)=2c+n-1.
$$
This means that $\{\gamma,\cdots,\gamma^k\}$ generates $\bigoplus_{*=2c-(n-1)}^{2c+n-1}SH_*^{S^1,+}(W_0)$. Since $\mu_{CZ}(\gamma^{k+1})>\mu_{CZ}(\gamma^k)$ for all $k\in\N$ and $SH^{S^1,+}_*(W_0)=0$ for all $2c+n-1<*<4c-(n-1)$ according to Proposition \ref{prop:SH for prequantization bundles}, we have
$$
\mu_{CZ}(\gamma^{k+1})=4c-(n-1).
$$
Since $SH_*^{S^1,+}(W_0)$ is periodic, i.e. 
$$
\bigoplus_{*=2c-(n-1)}^{2c+n-1}SH_*^{S^1,+}(W_0)\cong\bigoplus_{*=2Nc-(n-1)}^{2Nc+n-1}SH_*^{S^1,+}(W_0)
$$
for all $N\in\N$ and 
$$
SH_*^{S^1,+}(W_0)=0,\quad *\notin\bigcup_{N=\N}[2Nc-(n-1),2Nc+n-1],
$$
$\{\gamma^{(N-1)k+1},\cdots,\gamma^{Nk}\}$ generates $\bigoplus_{*=2Nc-(n-1)}^{2Nc+n-1}SH_*^{S^1,+}(W_0)$ for all $N\in\N$. In particular,
$$
\mu_{CZ}(\gamma^{(N-1)k+1})=2Nc-(n-1),\quad \mu_{CZ}(\gamma^{Nk})=2Nc+(n-1)
$$
for all $N\in\N$. Therefore we obtain
\begin{equation}\label{eq1}
2Nc=\mu_{CZ}(\gamma^{Nk+1})-\mu_{CZ}(\gamma)=rNk+\sum_{i=1}^j2[(Nk+1)\theta_i]
\end{equation}
since 
$$
2(N+1)c-(n-1)=\mu_{CZ}(\gamma^{Nk+1})=r(Nk+1)+\sum_{i=1}^j2[k\theta_i]+j.
$$
In particular if $N=1$, we have
\begin{equation}\label{eq2}
r=\frac{2c-\sum_{i=1}^j 2[(k+1)\theta_i]}{k}.
\end{equation}
Again by \eqref{eq1} and \eqref{eq2}, we have
$$
N\sum_{i=1}^j[(k+1)\theta_i]=\sum_{i=1}^j[(Nk+1)\theta_i],\quad \textrm{for all}\;\; N\in\N.
$$
But dividing out both sides by $N$ and taking a limit $N\to\infty$, we deduce
$$
\sum_{i=1}^j(k+1)\theta_i=\sum_{i=1}^jk\theta_i
$$
and this contradiction proves the theorem in the case $c\geq n$.

Now we consider the case $c\leq-n$, and we still assume by contradiction that $\gamma$ is the only simple closed Reeb orbit. Due to Proposition \ref{prop:index decreasing}, $\mu_{CZ}(\gamma^k)>\mu_{CZ}(\gamma^{k+1})$ and thus $\mu_{CZ}(\gamma)=2c+n-1$, see \eqref{eq:computation}. As above we have for $N\in\N$,
$$
\mu_{CZ}(\gamma^{(N-1)k+1})=2Nc+(n-1),\quad\mu_{CZ}(\gamma^{Nk})=2Nc-(n-1)
$$
and this case is proved in a similar fashion.
\hfill$\square$

\subsection{Brieskorn spheres}\quad\\[-1.5ex]

With the notation of the Introduction, consider the Brieskorn sphere $(\Sigma_a,\xi_a)$. The contact homologies of Brieskorn spheres were computed originally by \cite{Ust99} and reproved using the Morse-Bott approach by \cite{Bou02}. It is possible to compute the positive part of the $S^1$-equivariant symplectic homology of $V_\epsilon(a)$, a natural Weinstein filling of $(\Sigma_a,\xi_a)$ in a similar way or using an isomorphism between contact homology and the positive part of $S^1$-equivariant symplectic homology in \cite{BO12b}. Therefore we have if $a_0\equiv\pm1$ mod 8 and $a_0\neq 1$,
$$
SH_*^{S^1,+}(V_\epsilon(a))=\left\{\begin{array}{ll}
0\quad &*\in 2\Z+1\;\; or \;\; *<n-1,\\[1ex]
\Q\oplus\Q &*\in2\Big[\frac{2N}{a_0}\Big]+2N(n-2)+n+1,\;N\in\N,\;2N+1\notin a_0\Z,\\[1.2ex]
\Q & otherwise.
\end{array}\right.
$$
If $a_0=1$, 
$$
SH_*^{S^1,+}(V_\epsilon(a))=\left\{\begin{array}{ll}
\Q & *=n-1+2k,\;k\in\N,\\[1ex]
0\quad & otherwise.
\end{array}\right.
$$
\begin{Rmk}\label{rmk:brieskorn}
The contact homology computations of $(\Sigma_a,\xi_a)$ in \cite{Ust99,Bou02} are not correct when $a_0=1$ but their proof can be easily rectified. For instance, in p.105 of \cite{Bou02}, the case (ii) Action$=p\pi$ (when $z_0\neq0$)  includes the case (i) Action$=\pi$ (when $z_0=0)$ when $a_0=1$. In fact, $\Sigma_a$ with $a_0=1$ is a standard sphere (i.e. has a tight contact structure) and thus the computation of the $S^1$-equivariant symplectic homology agrees with Proposition \ref{prop:symplectic homology computation}.
\end{Rmk}

\noindent\textbf{Proof of Theorem C.} \\[-1.5ex]

The case that $a_0=1$ was treated in Theorem A, see Remark \ref{rmk:brieskorn}. We assume that $a_0\neq 1$. Since $SH^{S^1,+}(V_\epsilon(a))$ does not vanish, there exists a simple closed Reeb orbit $\gamma$. Suppose that there is no another simple closed Reeb orbit except $\gamma$.\\[-1.5ex]

\noindent\underline{Case 1}. If $\mu_{CZ}(\gamma)<n-1$, there exists a closed Reeb orbit $v$ with $\mu_{CZ}(v)=\mu_{CZ}(\gamma)+1$ or $\mu_{CZ}(v)=\mu_{CZ}(\gamma)-1$ since $SH^{S^1,+}_{\mu_{CZ}(\gamma)}(V_\epsilon(a))=0$. Since $v$ has to be a good closed Reeb orbit to contribute $SH_*^{S^1,+}$, it cannot be a multiple cover of $\gamma$.\\[-1.5ex]

\noindent\underline{Case 2}. If $\mu_{CZ}(\gamma)\geq n-1$, $\mu_{CZ}(\gamma)=n-1$ since $SH^{S^1,+}_{n-1}(V_\epsilon(a))=\Q$ and $\mu_{CZ}(\gamma^{k+1})\geq\mu_{CZ}(\gamma^k)$ for all $k\in\N$ due to Proposition \ref{prop:index nondecreasing}. Thus the iteration formula for the Conley-Zehnder index of $\gamma$ is 
$$
\mu_{CZ}(\gamma^k)=kr+\sum_{i=1}^j2[k\theta_i]+j,\quad r+j=n-1
$$
where $j\in\{0,\dots,n-1\}$. We claim that $r=0$. If $r\in 2\N$, every multiple cover of $\gamma$ is good and there exists $k_0\in\N$ such that $\mu_{CZ}(\gamma^{k_0+1})\geq\mu_{CZ}(\gamma^{k_0})+4$. This contradicts that $\dim SH^{S^1,+}_*(V_\epsilon(a))\geq1$ for every even degree greater than $n-2$. If $r\in 2\N-1$, only odd multiple covers are good and there exists $k_0\in2\N+1$ such that $\mu_{CZ}(\gamma^{k_0+2})\geq\mu_{CZ}(\gamma^{k_0})+4$. This is again a contradiction and proves the claim. Therefore the formula for the Conley-Zehnder index reduces to
$$
\mu_{CZ}(\gamma^k)=\sum_{i=1}^{n-1}2[k\theta_i]+n-1.
$$
We observe that, since $a_0$ is odd,
$$
\big\{N\in\N\,\big|\,2N+1\in a_0\Z\big\}=\Big\{\frac{(2\ell-1)a_0-1}{2}\,\Big|\,\ell\in\N\Big\}.
$$
Let $g$ be a nondecreasing bijective map
$$
g:\N\to\{N\in\N\,|\,2N+1\notin a_0\Z\}.
$$
In fact,
\bean
g(N)&=N+\max\Big\{\ell\in\N\,\Big|\,\frac{(2\ell-1)a_0-1}{2}\leq N\Big\}\\
&=N+\Big[\frac{2N+1}{2a_0}+\frac{1}{2}\Big].
\eea
Then $\dim SH_*^{S^1,+}(V_\epsilon(a))=2$ if and only if $*=f(g(N))$ for some $N\in\N$ where
$$
f(N)=2\Big[\frac{2N}{a_0}\Big]+2N(n-2)+n+1.
$$
Since $\mu_{CZ}(\gamma^{k+1})\geq\mu_{CZ}(\gamma^k)$ for all $k\in\N$, 
\bean
SH_{f(g(1))}^{S^1,+}(V_\epsilon(a))=\Q\langle\gamma^n,\gamma^{n+1}\rangle,\quad
SH_{f(g(2))}^{S^1,+}(V_\epsilon(a))=\Q\langle\gamma^{2n-1},\gamma^{2n}\rangle,\quad\cdots
\eea
More generally we have
$$
SH_{f(g(N))}^{S^1,+}(V_\epsilon(a))=\Q\langle\gamma^{h(N)},\gamma^{h(N)+1}\rangle
$$
for 
$$
h(N)=\frac{f(g(N))-(n-1)}{2}+N.
$$
In particular, for $i\in\N$, 
$$
f(g(a_0^2i))=2ia_0^2n+2ia_0n-4ia_0^2+n+1
$$
and thus
$$
h(a_0^2i)=i(a_0^2n-a_0^2+a_0n+2)+1.
$$
We abbreviate 
$$
\beta:=a_0^2n-a_0^2+a_0n+2.
$$
We have observed that
$$
\mu_{CZ}(\gamma^{\beta i+1})=\mu_{CZ}(\gamma^{\beta i+2}),\quad \textrm{for all}\;\;i\in\N.
$$
However since $\theta_1$ is irrational, there exists $i_*\in\N$ such that 
$$
(\beta i_*+1)\theta_1-[(\beta i_*+1)\theta_1]\approx 1
$$
and thus
\bean
\mu_{CZ}(\gamma^{\beta i_*+2})&=\sum_{i=1}^{n-1}2[(\beta i_*+2)\theta_i]+n-1\\
&\geq\sum_{i=1}^{n-1}2[(\beta i_*+1)\theta_i]+n+1\\
&=\mu_{CZ}(\gamma^{\beta i_*+1})+2.
\eea
This contradiction shows that $\gamma$ cannot generate all nonzero homology classes and hence the theorem is proved.
\hfill$\square$

\section{Appendix: More examples}
In this appendix we give examples which have two closed Reeb orbits even though they do not meet the requirements of Theorem A. For simplicity we treat 5-dimensional case, see Example \ref{ex:iteration formula for 4-dim}.

\begin{prop}
{\em Let $(\Sigma,\xi)$ be a contact 5-manifold which has displaceable exact contact embedding into $(W,\om=d\lambda)$ which is convex infinity and satisfies $c_1(W)|_{\pi_2(W)}=0$. Suppose that a corresponding contact form $\alpha$ is nondegenerate and $\alpha-\lambda|_\Sigma$ is exact. If $b_2(W_0,\Sigma;\Q)=1$ and $b_4(W_0,\Sigma;\Q)=0$, there are two closed Reeb orbits contractible in $W$.}
\end{prop}
\begin{proof}
We may assume that $b_3(W_0,\Sigma;\Q)=b_5(W_0,\Sigma;\Q)=0$ since otherwise the assertion is covered by Theorem A. According to Proposition \ref{prop:symplectic homology computation}, we have 
$$
SH_0^{S^1,+}(W_0)=SH_2^{S^1,+}(W_0)=\Q,\quad SH_{2\ell+2}^{S^1,+}(W_0)=\Q\oplus\Q,\quad \ell\in\N.
$$
and 
$$
SH_*^{S^1,+}(W_0)=0,\quad *\in\Z\setminus(2\N\cup\{0\}).
$$
Suppose that there exists precisely one simple closed Reeb orbit $\gamma$.  We recall that $|\mu_{CZ}(\gamma^k)-k\Delta(\gamma)|<n-1$ from \cite{SZ92} and this shows that $\Delta(\gamma)>0$ since multiple covers of $\gamma$ have to generate all nonzero homology classes. One can immediately see that $\gamma$ cannot be hyperbolic, see Example \ref{ex:iteration formula for 4-dim}. If $\gamma$ is not elliptic nor hyperbolic, it has to be
$$
\mu_{CZ}(\gamma^k)=-k+2[k\theta]+1,\quad k\in\N,\;\theta\in(0,1)\setminus\Q.
$$
But Corollary \ref{cor:resonance identity in the displaceable case} implies 
$$
1=\frac{1}{2\Delta(\gamma)}=\frac{1}{2(2\theta-1)}
$$
which implies a contradiction $\theta=3/4$. The remaining case is that $\gamma$ is elliptic and 
$$
\mu_{CZ}(\gamma^k)=-2k+2[k\theta_1]+2[k\theta_2]+2,\quad k\in\N,\;\theta_i\in(0,1)\setminus\Q.
$$
 Again by Corollary \ref{cor:resonance identity in the displaceable case}, we obtain
$$
\theta_1+\theta_2=\frac{3}{2}.
$$
Since both $\theta_1$ and $\theta_2$ are irrational, we have 
$$
[2k\theta_1]+[2k\theta_2]=[2k\theta_1]+[3k-2k\theta_1]=3k-1,\quad k\in\N
$$
and thus $\mu_{CZ}(\gamma^{2k})=2k$. From this we can derive $\mu_{CZ}(\gamma^{2k+1})=2k+2$ for all $k\in\N$ since $|\mu_{CZ}(\gamma^{k+1})-\mu_{CZ}(\gamma^k)|\leq2$. This yields that
$$
[(2k+1)\theta_1]+[(2k+1)\theta_2]=[(2k+1)\theta_1]+[3k+1+1/2-(2k+1)\theta_1]=3k+1,
$$
and thus we have the following contradictory inequality.
$$
(2k+1)\theta_1-[(2k+1)\theta_1]<\frac{1}{2},\quad k\in\N.
$$
Indeed, since $\theta_1$ is irrational, the sequence $\{(2k+1)\theta_1 \textrm{ mod } 1\,|\,k\in\N\}$ is dense in $[0,1]$.
\end{proof}

\begin{prop}
{\em Let $(\Sigma,\xi)$ be a contact 5-manifolds which has displaceable exact contact embedding into $(W,\om=d\lambda)$ which is convex infinity and satisfies $c_1(W)|_{\pi_2(W)}=0$. Suppose that a corresponding contact form $\alpha$ is nondegenerate and $\alpha-\lambda|_\Sigma$ is exact. If $b_2(W_0,\Sigma;\Q)=1$ and $b_4(W_0,\Sigma;\Q)\geq 4$, there are two closed Reeb orbits contractible in $W$.}
\end{prop}
\begin{proof}
By the same reason as above we may assume that $b_3(W_0,\Sigma;\Q)=b_5(W_0,\Sigma;\Q)=0$. According to Proposition \ref{prop:symplectic homology computation}, we have 
$$
SH_0^{S^1,+}(W_0)=\Q,\quad SH_2^{S^1,+}(W_0)=\Q^{\oplus(b_4(W_0,\Sigma;\Q)+1)},\quad SH_{2\ell+2}^{S^1,+}(W_0)=\Q^{\oplus(b_4(W_0,\Sigma;\Q)+2)}
$$
for all $\ell\in\N$ and 
$$
SH^{S^1,+}_*(W_0)=0,\quad *\in\Z\setminus(2\N\cup\{0\}).
$$
Assume by contradiction that there exists precisely one simple closed Reeb orbit $\gamma$. As in the proof of the previous proposition, $\Delta(\gamma)>0$. The case that $\gamma$ is not elliptic can be easily excluded. If $\gamma$ is hyperbolic, $\mu_{CZ}(\gamma^k)=k\mu_{CZ}(\gamma)$ and multiple covers of $\gamma$ cannot generate all nontrivial homology classes. Suppose that $\gamma$ is not elliptic nor hyperbolic. Then 
$$
\mu_{CZ}(\gamma^k)=rk+2[k\theta]+1,\quad k\in\N,\;\theta\in(0,1)\setminus\Q.
$$
for some $r\in\Z$. Due to the homology computation, $r=-1$ and  $\mu_{CZ}(\gamma^k)=-k+2[k\theta]+1$.  But by Corollary \ref{cor:resonance identity in the displaceable case}, we have
$$
\frac{b_4(W_0,\Sigma;\Q)+2}{2}=\frac{1}{2\Delta(\gamma)}=\frac{1}{2(2\theta-1)}
$$
and this contradicts $\theta\notin\Q$.
The only possible nontrivial case is that $\mu_{CZ}(\gamma)=0$ and
$$
\mu_{CZ}(\gamma^k)=-2k+2[k\theta_1]+2[k\theta_2]+2,\quad\theta_i\in(0,1)\setminus\Q,\; k\in\N.
$$
 Corollary \ref{cor:resonance identity in the displaceable case} yields that
$$
\theta_1+\theta_2=\frac{b_4(W_0,\Sigma;\Q)+3}{b_4(W_0,\Sigma;\Q)+2}\leq\frac{7}{6}.
$$
We note that $\mu_{CZ}(\gamma^k)\geq2$ for $k\geq2$ because of $SH_0^{S^1,+}(W_0)=\Q$ and $SH_*^{S^1,+}(W_0)=0$ for all $*<0$. Since $\mu_{CZ}(\gamma^2)\geq2$, both $\theta_1$ and $\theta_2$ are bigger than $1/2$. In addition $\mu_{CZ}(\gamma^3)\geq2$ implies that one of $\theta_1$ and $\theta_2$ is bigger than $2/3$. Thus we deduce
$$
\theta_1+\theta_2> \frac{2}{3}+\frac{1}{2}=\frac{7}{6}.
$$
This contradiction completes the proof.
\end{proof}


\end{document}